\documentclass[a4paper,11pt,oneside]{article}
\usepackage{amsthm,amsmath,amsfonts,amssymb,pstricks-add}
\usepackage[normalem]{ulem}

\newcommand{\Z}{\mathbb{Z}}

\newcommand{\pres}[2]{\langle {#1}\ |\ {#2} \rangle}

\newtheorem{theorem}{Theorem}
\newtheorem{lemma}[theorem]{Lemma}
\newtheorem{corollary}[theorem]{Corollary}
\newtheorem{observation}[theorem]{Observation}
\newtheorem{example}[theorem]{Example}
\newtheorem{remark}[theorem]{Remark}
\newtheorem{maintheorem}{Theorem}

\numberwithin{theorem}{section}

\begin{document}
\title{Hyperbolicity of T(6) Cyclically Presented Groups}
\author{Ihechukwu Chinyere and Gerald Williams\thanks{This work was supported by the Leverhulme Trust Research Project Grant RPG-2017-334. Part of the work was carried out during a research visit funded by Leverhulme Trust Emeritus Fellowship EM-2018-023$\backslash$9.}}

\maketitle

\begin{abstract}
We consider groups defined by cyclic presentations where the defining word has length three and the cyclic presentation satisfies the T(6) small cancellation condition. We classify when these groups are hyperbolic. When combined with known results, this completely classifies the hyperbolic T(6) cyclically presented groups.
\end{abstract}

\noindent \textbf{Keywords:} cyclically presented group, hyperbolic group, small cancellation theory.\newline
\noindent \textbf{MSC:} 20F05, 20F06, 20F67.

\section{Introduction}\label{sec:intro}

Groups defined by presentations that satisfy the $C(p)-T(q)$ (non-metric) small cancellation conditions where $1/p+1/q<1/2$ are hyperbolic (\cite[Corollary~3.3]{GerstenShortI}). Therefore the cases $(p,q)=(3,6),(4,4),(6,3)$ present boundary cases and here both hyperbolic and non-hyperbolic groups can arise. For these cases,
in~\cite[Corollary, page~1860]{IvanovSchupp} the $C(p)-T(q)$ presentations that define hyperbolic groups are characterised as those for which there is no minimal flat over the presentation. In this article we consider groups defined by a class of presentations that admit a certain cyclic symmetry and satisfy $C(3)-T(6)$. We classify when the corresponding groups are hyperbolic in terms of the defining parameters of the presentations.

The \em cyclically presented group \em $G_n(w)$ is the group defined by the \em cyclic presentation \em
\[ P_n(w) = \pres{x_0,\ldots, x_{n-1}}{w,\theta(w),\ldots, \theta^{n-1}(w)}\]
where $w(x_0,\ldots ,x_{n-1})$ is a cyclically reduced word in the free group $F_n$ (of length $l(w)$) with generators $x_0,\ldots ,x_{n-1}$ and $\theta : F_n\rightarrow F_n$ is the \em shift automorphism \em of $F_n$ given by $\theta(x_i)=x_{i+1}$ for each $0\leq i<n$ (subscripts mod~$n$, $n\geq 2$).

If a presentation satisfies T(6) then, as observed by Pride (\cite[Section~5]{PrideStarComplexes},\cite[Lemma~3.1]{GerstenShortI}), every piece has length~1 and so if $P_n(w)$ satisfies T(6) then it satisfies C($l(w)$)-T(6). Thus if $l(w)>3$ then the presentation $P_n(w)$ satisfies C(4)-T(6) and hence $G_n(w)$ is hyperbolic by~\cite[Corollary~4.1]{GerstenShortI} and therefore it is non-elementary hyperbolic by~\cite{Collins73} or \cite{EdjvetHowie}. If the length $l(w)=1$ then $G_n(w)$ is trivial, and if $l(w)=2$ then $G_n(w)$ is the free product of copies of $\Z$ or $\Z_2$.
Therefore we must consider the case $l(w)=3$ (in which case the $C(3)-T(6)$ condition coincides with the $T(6)$ condition). If $w$ is a positive (or negative) word, then we may assume $w=x_0x_kx_{l}$, and if $w$ is non-positive (and non-negative) then we may assume $w=x_0x_mx_{k}^{-1}$. Our main results consider these cases.

The groups $G_n(x_0x_mx_{k}^{-1})$ are known as the \em groups of Fibonacci-type \em and were introduced independently in~\cite{JohnsonMawdesley} and~\cite{CHR}, for algebraic and topological reasons, respectively. This family of groups contains the \em Fibonacci groups \em $F(2,n)=G_n(x_0x_1x_2^{-1})$ of~\cite{Conway65}, the \em Sieradski groups \em $S(2,n)=G_n(x_0x_2x_1^{-1})$ of~\cite{Sieradski}, and the \em Gilbert-Howie groups \em $H(n,m)=G_n(x_0x_mx_1^{-1})$ of~\cite{GilbertHowie}. They have been subsequently studied in~\cite{BardakovVesnin,W-CHR,COS,HowieWilliams,HowieWilliamsMFD} -- see \cite{Wsurvey} for a survey. In particular, the T(6) and T(7) presentations $P_n(x_0x_mx_{k}^{-1})$ were classified in~\cite[Theorem~10]{HowieWilliams} (see Corollary~\ref{cor:T6T7nonpositive}, below) and \cite[Theorem~11]{HowieWilliams} records that in the T(7) case the groups $G_n(x_0x_mx_{k}^{-1})$ are non-elementary hyperbolic. The groups $G_n(x_0x_kx_{l})$ were introduced in~\cite{CRS}, and studied further in~\cite{EdjvetWilliams,BogleyWilliamsCoherence,MohamedWilliams}. The T(6) presentations $P_n(x_0x_kx_l)$ were classified in~\cite[Lemma~5.1]{EdjvetWilliams} (see Lemma~\ref{lem:positiveT6}, below). Moreover, Theorem~3.7 of~\cite{MohamedWilliams} shows that for all but finitely many $n$ the T(6) groups $G_n(x_0x_kx_l)$ are hyperbolic.

Our main results are as follows:

\begin{maintheorem}\label{thm:positive}
Let $n\geq 2$, $0\leq k,l<n$, $(n,k,l)=1$ and suppose that the cyclic presentation $P_n(x_0x_kx_{l})$ is T(6) and let $G=G_n(x_0x_kx_l)$.
If $n=7$ or $8$ or
\begin{itemize}
  \item[(a)] $n=21$ and ($l\equiv 5k$ or $k\equiv 5l$~mod~$n$); or
  \item[(b)] $n=24$ and ($l\equiv 5k$ or $k\equiv -4l$ or $l\equiv -4k$~mod~$n$);
\end{itemize}
then $G$  is not hyperbolic; otherwise $G$ is non-elementary hyperbolic.
\end{maintheorem}

\begin{maintheorem}\label{thm:nonpositive}
Let $n\geq 2$, $0\leq m,k<n$, $(n,m,k)=1$, $m\neq k$, $k\neq 0$, and suppose that the cyclic presentation $P_n(x_0x_mx_k^{-1})$ is T(6) and let $G=G_n(x_0x_mx_k^{-1})$. If $n=8$ or ($n\geq 12$ is even and $2(2k-m)\equiv 0$~mod~$n$) then $G$ is not hyperbolic; otherwise $G$ is non-elementary hyperbolic.
\end{maintheorem}

The coprimality hypotheses of Theorems~\ref{thm:positive},\ref{thm:nonpositive} are imposed to avoid the presentations and groups decomposing in canonical ways. Specifically, if $d=(n,k,l)>1$ then the presentation $P_n(x_0x_kx_{l})$ is the disjoint union of $d$ copies of the presentation $P_{n/d}(x_0x_{k/d}x_{l/d})$ (\cite[Lemma~2.4]{CRS}) and so $G_n(x_0x_kx_{l})$ is the free product of $d$ copies of $G_{n/d}(x_0x_{k/d}x_{l/d})$ and $P_n(x_0x_kx_{l})$ satisfies T(6) if and only if  $P_{n/d}(x_0x_{k/d}x_{l/d})$ satisfies $T(6)$. Similarly if $d=(n,m,k)>1$ then $P_n(x_0x_mx_k^{-1})$ is the disjoint union of $d$ copies of $P_{n/d}(x_0x_{m/d}x_{k/d}^{-1})$ (\cite[Lemma~1.2]{BardakovVesnin}), so the analogous conclusions can be drawn in this case. Since a free product $H*K$ is hyperbolic if and only if $H$ and $K$ are hyperbolic -- see, for example, \cite[Theorem~H]{BaumslagGerstenShapiroShort} -- there is no loss in generality in assuming that such decompositions do not arise. The conditions $m\neq k, k\neq 0$ are imposed in Theorem~\ref{thm:nonpositive} to ensure that the relators are cyclically reduced and, as otherwise, the group is trivial.

A relator of a presentation is \em freely redundant \em if it is freely equal to the conjugate of another relator or its inverse. The cyclic presentations $P_n(x_0x_mx_k^{-1})$ have no freely redundant relators and if $(n,k,l)=1$ the presentation $P_n(x_0x_kx_l)$ has a freely redundant relator if and only if $n=3$ and $\{k,l\}=\{1,2\}$, in which case it defines the (non-elementary hyperbolic) group $\Z*\Z$. For this reason, throughout this article we will only consider presentations with no freely redundant relators.

A consequence of Theorem~\ref{thm:positive} and the results of~\cite{EdjvetWilliams} (see also~\cite[Example~10, Propositions~7.7,7.8]{MohamedWilliams}) is that the hyperbolicity status of cyclically presented groups $G$ of the form $G_n(x_0x_kx_l)$ is known, except when $G$ is isomorphic to $G_n(x_0x_1x_{n/2-1})$ for even $n\geq 10$, $n\neq 12,18$.

We prove Theorem~\ref{thm:positive} in Section~\ref{sec:positivecase} and Theorem~\ref{thm:nonpositive} in Section~\ref{sec:nonpositivecase}.

\section{The positive case}\label{sec:positivecase}

Let $P=\pres{\mathcal{X}}{\mathcal{R}}$ be a group presentation such that each relator $r\in \mathcal{R}$ is a cyclically reduced word in the generators. Let $\tilde{\mathcal{R}}$ denote the symmetrized closure of $\mathcal{R}$; that is, the set of all cyclic permutations of elements in $\mathcal{R}\cup \mathcal{R}^{-1}$. The \em star graph \em of $P$ is the undirected graph with vertex set $\mathcal{X} \cup \mathcal{X}^{-1}$, and with an edge joining vertices $x,y$ for each word $xy^{-1}u$ in $\tilde{\mathcal{R}}$. These words occur in pairs: $xy^{-1}u\in \tilde{\mathcal{R}}$ implies that $yx^{-1}u^{-1}\in\tilde{\mathcal{R}}$. Such pairs are called \em inverse pairs \em and the two corresponding edges are identified in $\Gamma$~\cite[page~61]{LyndonSchupp}. Thus if $\Gamma$ is the star graph of the cyclic presentation $P_n(x_0x_kx_l)$ then $\Gamma$ has vertices $x_i,x_i^{-1}$ and edges $x_i-x_{i+k}^{-1}$, $x_i-x_{i+l-k}^{-1}$, $x_i-x_{i-l}^{-1}$ ($0\leq i<n$), which we will refer to as edges of type $X,Y,Z$, respectively.

By~\cite{HillPrideVella} a presentation in which each relator has length at least 3 satisfies T($q$) ($q>3$) if and only if its star graph has no cycle of length less than~$q$. As we are interested in presentations that satisfy T(6), in Section~\ref{sec:positiveshortcycles} we analyze cycles of length at most 6 in the star graph $\Gamma$ of $P_n(x_0x_kx_l)$. In particular, we note that $\Gamma$ always contains a cycle of length at most 6. We show that if two additional cycle types of length 6 arise then only a few small values of $n$ are possible and $G_n(x_0x_kx_l)$ is isomorphic to one of only a few groups, one of which turns out to be hyperbolic. In Section~\ref{sec:nonhyperbolicgroupspositive} we prove that the remainder are not hyperbolic. In Section~\ref{sec:curvaturepositive} we consider the case when at most one further cycle type of length 6 occurs and perform a detailed analysis of van Kampen diagrams (see~\cite[Chapter~5]{LyndonSchupp}) over the defining presentation to prove that  $G_n(x_0x_kx_l)$ has a linear isoperimetric function, and hence is hyperbolic. We then combine these results to prove Theorem~\ref{thm:positive} in Section~\ref{sec:proofofpositive}.

\begin{table}
\begin{center}
\begin{tabular}{|cr|lll|}\hline
& $j$ & \textbf{0} & \textbf{1} & \textbf{2}\\\hline
$(B.j)$ & congruence &$2k-l\equiv0$ & $2l-k\equiv0$ & $k+l\equiv0$\\
& cycle type& $XY$ & $YZ$ &$ZX$\\\hline
$(C.j)$ & congruence & $l\equiv \pm \frac{n}{3}$ & $k\equiv \pm \frac{n}{3}$ & $k-l\equiv \pm \frac{n}{3}$\\
& cycle type& $XZYZ$  & $YXZX$ & $ZYXY$\\\hline
$(D.j)$ & congruence & $2k-l\equiv \frac{n}{2}$ & $2l-k\equiv \frac{n}{2}$ & $k+l\equiv \frac{n}{2}$\\
& cycle type& $(XY)^2$ or & $(YZ)^2$ or&$(ZX)^2$ or\\
&                 & $XYZYXZ$ & $YZXZYX$ & $ZXYXZY$\\\hline
$(E.j)$ & congruence & $2k-l\equiv \pm \frac{n}{3}$ & $2l-k\equiv \pm \frac{n}{3}$ & $k+l\equiv \pm \frac{n}{3}$\\
& cycle type & $(XY)^3$ & $(YZ)^3$ & $(ZX)^3$\\\hline
$(F1.j)$ & congruence & $5k-l\equiv0$ & $5l-4k\equiv0$ & $k+4l\equiv0$\\
 & cycle type& $(XY)^2XZ$ & $(YZ)^2YX$ & $(ZX)^2ZY$\\\hline
$(F2.j)$ & congruence & $4l-5k\equiv0$ & $4k+l\equiv0$ & $5l-k\equiv0$\\
 & cycle type& $(YX)^2YZ$ & $(XZ)^2XY$ & $(ZY)^2ZX$\\\hline
\end{tabular}
\end{center}
\caption{Congruences (mod~$n$) corresponding to short cycles in the star graph of $P_n(x_0x_kx_l)$.\label{tab:BCDEFG}}
\end{table}

\subsection{Analysis of short cycles in the star graph of $P_n(x_0x_kx_l)$}\label{sec:positiveshortcycles}

The following classification of the T(6) cyclic presentations $P_n(x_0x_kx_l)$ in terms of three types of congruences (B),(C),(D) was obtained in~\cite{EdjvetWilliams}. As indicated in Table~\ref{tab:BCDEFG}, the (B) conditions correspond to cycles (of length 2) of the form $XY$,$YZ$,$ZX$; the (C) conditions correspond to cycles (of length 4) of the form $XZYZ$, $YXZX$, $ZYXY$; and the (D) conditions correspond to cycles (of length 4) of the form $(XY)^2$, $(YZ)^2$, $(ZX)^2$, as well as to cycles (of length 6) of the form $XYZYXZ$, $YZXZYX$, $ZXYXZY$. Replacing parameter $k$ by $l-k$ and $l$ by $-k$ corresponds to replacing edge type $X$ by $Y$, $Y$ by $Z$, and $Z$ by $X$ and to replacing a condition $(*.j)$ of Table~\ref{tab:BCDEFG} by $(*.j+1)$ (mod~$3$), and replacing the group $G_n(x_0x_kx_l)$ by the isomorphic copy $G_n(x_0x_{l-k}x_{-k})$. (To see that $G_n(x_0x_kx_l)\cong G_n(x_0x_{l-k}x_{-k})$ set $j=i+k$ in the relators $x_ix_{i+k}x_{i+l}$ of $P_n(x_0x_kx_l)$ and then cyclically permute to get the relators $x_jx_{j+l-k}x_{j-k}$ of $P_n(x_0x_{l-k}x_{-k})$.) Replacing parameter $k$ by $l-k$ corresponds to interchanging the roles of edge types $X$ and $Y$ and so interchanging the roles of conditions $(F1.j)$ and $(F2.j)$, and replacing the group $G_n(x_0x_kx_l)$ by the isomorphic copy $G_n(x_0x_{l-k}x_l)$. (To see that $G_n(x_0x_kx_l)\cong G_n(x_0x_{l-k}x_l)$ replace the generators $x_i$ by $x_i^{-1}$, negate the subscripts, and set $j=-i-l$ in the relators $x_ix_{i+k}x_{i+l}$ and then invert to get the relators $x_jx_{j+l-k}x_{j+l}$ of $P_n(x_0x_{l-k}x_l)$.)

\begin{lemma}[{\cite[Lemma~5.1]{EdjvetWilliams}}]\label{lem:positiveT6}
Let $n\geq 2$ and suppose that $(n,k,l)=1$, $0\leq k,l<n$. Then $P_n(x_0x_kx_l)$ satisfies T(6) if and only if none of the congruences $(B.j)$, $(C.j)$, $(D.j)$ ($0\leq j\leq 2$)  of Table~\ref{tab:BCDEFG} hold.
\end{lemma}

\begin{observation}[{See~\cite[Theorem~3.4]{MohamedWilliams}}]\label{observation:XYZXYZcycle}
Suppose that $(n,k,l)=1$, $0\leq k,l<n$ and that none of the congruences $(B.j)$,$(C.j)$,$(D.j)$ ($0\leq j\leq 2$) of Table~\ref{tab:BCDEFG} hold. Then for each $0\leq i<n$ the sequence of vertices and edges $x_i-x_{i+k}^{-1}-x_{i+2k-l}-x_{i+2k-2l}^{-1}-x_{i+k-2l}-x_{i-l}^{-1}-x_i$ forms a cycle of length~6 of the form $(XYZ)^2$ in the star graph $\Gamma$.
\end{observation}

We now consider how other cycles of length $6$ can arise in $\Gamma$.

\begin{lemma}\label{lem:6cycle}
Let $n\geq 2$ and suppose that $(n,k,l)=1$, $0\leq k,l<n$ and that none of the congruences $(B.j)$,$(C.j)$,$(D.j)$ ($0\leq j\leq 2$) of Table~\ref{tab:BCDEFG} hold.
Then the star graph $\Gamma$ contains a cycle of length 6 of cycle type other than $(XYZ)^2$ if and only if at least one of the congruences $(E.j)$,$(F1.j)$,$(F2.j)$ ($0\leq j\leq 2$) of Table~\ref{tab:BCDEFG} holds, in which case the corresponding entry of the table is a label of the cycle.
\end{lemma}

\begin{proof}
Let $C$ be a cycle of length 6 in $\Gamma$. Then there are no subpaths of $C$ of the form $XX$, $YY$ or $ZZ$. If $C$ involves each of the edge types $X,Y,Z$ twice then $C$ is a cycle of the form $(XYZ)^2$, $XYZYXZ$, $YZXZYX$, or $ZXYXZY$. But these last three cycles only occur if the congruence $(D.j)$ holds, contrary to hypothesis.
If $C$ does not involve an edge of type $X$ (resp.\,$Y$, resp.\,$Z$) then $C$ is a cycle of the form $(YZ)^3$ (resp.\,$(XZ)^3$, resp.\,$(XY)^3$), which correspond to the conditions $(E.j)$.
If $C$ involves exactly one edge of type $X$ (resp.\,$Y$, resp.\,$Z$) then $C$ is a cycle of the form $(YZ)^2YX$ or $(ZY)^2ZX$ (resp.\,$(ZX)^2ZY$ or $(XZ)^2XY$, resp.\,$(XY)^2XZ$ or $(YX)^2YZ$), which correspond to the conditions $(F1.j)$ or $(F2.j)$.

Conversely, if any of the congruences $(E.j),(F1.j), (F2.j)$ hold then the corresponding entry of Table~\ref{tab:BCDEFG} is the label of a cycle of length~6 in $\Gamma$.
\end{proof}

\begin{lemma}\label{lem:combinationsofEFG}
Let $n\geq 2$ and suppose that $(n,k,l)=1$, $0\leq k,l<n$ and that none of the congruences $(B.j)$,$(C.j)$,$(D.j)$ ($0\leq j\leq 2$) hold. If more than one of the congruences $(E.j)$,$(F1.j)$, $(F2.j)$ ($0\leq j\leq 2$) hold then one of the following holds:
\begin{itemize}
\item[(a)] $n=7$ and ($l\equiv 5k$ or $k\equiv 5l$~mod~$n$);
\item[(b)] $n=8$ and ($l\equiv 5k$ or $k\equiv 5l$~mod~$n$);
\item[(c)] $n=21$ and ($l\equiv 5k$ or $k\equiv 5l$~mod~$n$);
\item[(d)] $n=24$ and ($l\equiv 5k$ or $k\equiv -4l$ or $l\equiv -4k$~mod~$n$);
\item[(e)] $n=27$ and ($l\equiv 5k$ or $k\equiv 5l$ or $4k\equiv 5l$ or $4l\equiv 5k$ or $k\equiv -4l$ or $l\equiv -4k$~mod~$n$).
\end{itemize}
In each case $G\cong G_n(x_0x_1x_5)$.
\end{lemma}

\begin{proof}
(Throughout this proof, the $j$ value in a condition $(*.j)$ is to be taken mod~$3$.)
If $(E.j)$ and $(F1.j)$ hold then $(B.j+1)$ holds, a contradiction.
If $(E.j)$ and $(F1.j+1)$ hold then $(B.j+2)$ holds.
If $(E.j)$ and $(F2.-j)$ hold then $(B.j+2)$ holds.
If $(E.j)$ and $(F2.1-j)$ hold then $(B.j+1)$ or $(D.j+1)$ holds.
If $(F1.j)$ and $(F2.-j)$ hold then $(C.j)$ holds.
If $(F1.j)$ and $(F2.1-j)$ hold then $(B.j+1)$ holds.
Suppose now that any two of the $(E.j)$ conditions hold; then all three of them hold. Since $(B.0)$ does not hold condition $(E.0)$ implies $2k-l\equiv \pm n/3$~mod~$n$ and since $(B.2)$ does not hold condition $(E.2)$ implies $k+l\equiv \pm n/3$~mod~$n$. Thus $2k-l\equiv \epsilon (k+l)$ where $\epsilon =\pm 1$. If $\epsilon =+1$ then $(B.1)$ holds, a contradiction; and if $\epsilon =-1$ then $(C.0)$ or $(C.1)$ holds, a contradiction.

Suppose that two of the $(F1.j)$ conditions hold. Then all of them hold so, in particular, $l\equiv 5k$~mod~$n$. Summing the congruences $(F1.0)$ and $(F1.1)$ gives that $k\equiv -4l$~mod~$n$ and so (by $(F1.0)$) $21l\equiv 0$~mod~$n$. Moreover $1=(n,k,l)=(n,-4l,l)=(n,l)$ so $n|21$. If $n=3$ then $(F1.0)$ implies that $(B.0)$ holds, so $n=7$ or $21$. An analogous argument shows that if two of the $(F2.j)$ conditions hold then $k\equiv 5l$ and $n=7$ or $21$, thus giving cases (a),(c) of the statement.

Suppose $(F1.j)$ and $(F2.2-j)$ hold. We claim that $n=8$ or $24$; it then follows from one of the congruences that $l\equiv 5k$ or $k\equiv 5l$~mod~$n$ (by multiplying by 5, if necessary), giving cases (b),(d). We prove this in the case $(F1.0)$ and $(F2.2)$, the other cases being similar. The congruence $(F1.0)$ implies $l\equiv 5k$~mod~$n$, so substituting into $(F2.2)$ gives $24k\equiv 0$~mod~$n$, but $1=(n,k,l)=(n,k)$ so $n|24$. If $n\leq 6$ then some condition $(B.j),(C.j),(D.j)$ holds, and if $n=12$ then $(B.2)$ or $(D.2)$ holds, a contradiction; thus $n=8$ or $24$, as claimed.

Suppose that either ($(E.j)$ and $(F1.j+2)$) or ($(E.j)$ and $(F2.2-j)$) hold. We claim that $n=27$. We prove this in the case where $(E.0)$ and $(F1.2)$ hold, the other cases being proved analogously. The congruence $(F1.2)$ implies $k\equiv -4l$~mod~$n$ so $(E.0)$ implies $27l\equiv 0$~mod~$n$, but $1=(n,k,l)=(n,l)$ so $n|27$. If $n=3$ or $9$ then $(B.0)$ holds, and hence $n=27$, as claimed.

The final assertion that $G_n(x_0x_kx_l)\cong G_n(x_0x_1x_5)$ in each case follows from~\cite[Lemma~2.1]{EdjvetWilliams}.
\end{proof}

We now deal with the group arising in case~(e) of Lemma~\ref{lem:combinationsofEFG}.

\begin{example}[The group $G_{27}(x_0x_1x_5)$]\label{ex:n=27}
\em Using KBMAG~\cite{KBMAG} it is straightforward to show that the group $G_{27}(x_0x_1x_5)$ is hyperbolic, and since it contains a non-abelian free subgroup (by~\cite[Corollary~5.2]{EdjvetWilliams}), it is non-elementary hyperbolic.\em
\end{example}

\subsection{Non-hyperbolic groups $G_n(x_0x_kx_l)$}\label{sec:nonhyperbolicgroupspositive}

In this section we consider the groups arising in cases (a)-(d) of Lemma~\ref{lem:combinationsofEFG}. First we recall that the group $G_7(x_0x_1x_5)$ is not hyperbolic; see~\cite[Example~3.8]{MohamedWilliams} for a discussion.

\begin{lemma}[{\cite{EdjvetVdovina,Cartwrightetal}}]\label{lem:n=7nonhyp}
The group $G_7(x_0x_1x_5)$ is not hyperbolic.
\end{lemma}

We now show that the group $G_8(x_0x_1x_5)$ is not hyperbolic. We do this by an application of the Flat Plane Theorem~\cite{Bridson} (an alternative approach would be to use~\cite[Corollary, page 1860]{IvanovSchupp}).

\begin{lemma}\label{lem:n=8nonhyp}
The group $G_8(x_0x_1x_5)$ is not hyperbolic.
\end{lemma}

\begin{proof}
Since the presentation $P_8(x_0x_1x_5)$ satisfies C(3)-T(6) and each relator has length~3, each face in the geometric realisation $\tilde{C}$ of the Cayley complex of $P$ (obtained by assigning length~1 to each edge) is an equilateral triangle, and so $\tilde{C}$ satisfies the CAT(0) inequality. Consider the geometric realisation $\Delta_0$ of the reduced van Kampen diagram given in Figure~\ref{fig:positiveflat} and for each $0\leq i<n$ let $\Delta_i$ be obtained from $\Delta_0$ by applying the shift $\theta^i$ to each edge. Then placing $\Delta_0,\Delta_2,\Delta_4,\Delta_6$ one above the other gives the geometric realisation $\Delta$ of a reduced van Kampen diagram. Copies of $\Delta$ tile the Euclidean plane without cancellation of faces. Thus there is an isometric embedding of the Euclidean plane in $\tilde{C}$, and so the result follows from the Corollary to Theorem~A in~\cite{Bridson}.
\end{proof}

\begin{figure}
\begin{center}
\psset{xunit=1.1cm,yunit=1.1cm,algebraic=true,dimen=middle,dotstyle=o,dotsize=7pt 0,linewidth=1.6pt,arrowsize=3pt 2,arrowinset=0.25}

\begin{pspicture*}(0.5,-2.2)(11.5,2)
\psline[linewidth=2.pt,ArrowInside=->,ArrowInsidePos=0.5](2.,0.)(1.,1.71)
\psline[linewidth=2.pt,ArrowInside=->,ArrowInsidePos=0.5](1.,1.71)(3.,1.712)
\psline[linewidth=2.pt,ArrowInside=->,ArrowInsidePos=0.5](3.,1.712)(2.,0.)
\psline[linewidth=2.pt,ArrowInside=->,ArrowInsidePos=0.5](3.,-1.71)(2.,0.)
\psline[linewidth=2.pt,ArrowInside=->,ArrowInsidePos=0.5](2.,0.)(4.,0.)
\psline[linewidth=2.pt,ArrowInside=->,ArrowInsidePos=0.5](4.,0.)(3.,1.712)
\psline[linewidth=2.pt,ArrowInside=->,ArrowInsidePos=0.5](3.,1.712)(5.,1.71)
\psline[linewidth=2.pt,ArrowInside=->,ArrowInsidePos=0.5](5.,1.71)(4.,0.)
\psline[linewidth=2.pt,ArrowInside=->,ArrowInsidePos=0.5](4.,0.)(6.,0.)
\psline[linewidth=2.pt,ArrowInside=->,ArrowInsidePos=0.5](6.,0.)(5.,1.71)
\psline[linewidth=2.pt,ArrowInside=->,ArrowInsidePos=0.5](5.,1.71)(7.,1.71)
\psline[linewidth=2.pt,ArrowInside=->,ArrowInsidePos=0.5](7.,1.71)(6.,0.)
\psline[linewidth=2.pt,ArrowInside=->,ArrowInsidePos=0.5](6.,0.)(8.,0.)
\psline[linewidth=2.pt,ArrowInside=->,ArrowInsidePos=0.5](8.,0.)(7.,1.71)
\psline[linewidth=2.pt,ArrowInside=->,ArrowInsidePos=0.5](7.,1.71)(9.,1.71)
\psline[linewidth=2.pt,ArrowInside=->,ArrowInsidePos=0.5](9.,1.71)(8.,0.)
\psline[linewidth=2.pt,ArrowInside=->,ArrowInsidePos=0.5](8.,0.)(10.,0.)
\psline[linewidth=2.pt,ArrowInside=->,ArrowInsidePos=0.5](10.,0.)(9.,1.71)
\psline[linewidth=2.pt,ArrowInside=->,ArrowInsidePos=0.5](4.,0.)(3.,-1.71)
\psline[linewidth=2.pt,ArrowInside=->,ArrowInsidePos=0.5](3.,-1.71)(5.,-1.71)
\psline[linewidth=2.pt,ArrowInside=->,ArrowInsidePos=0.5](5.,-1.71)(4.,0.)
\psline[linewidth=2.pt,ArrowInside=->,ArrowInsidePos=0.5](6.,0.)(5.,-1.71)
\psline[linewidth=2.pt,ArrowInside=->,ArrowInsidePos=0.5](5.,-1.71)(7.,-1.71)
\psline[linewidth=2.pt,ArrowInside=->,ArrowInsidePos=0.5](7.,-1.71)(6.,0.)
\psline[linewidth=2.pt,ArrowInside=->,ArrowInsidePos=0.5](8.,0.)(7.,-1.71)
\psline[linewidth=2.pt,ArrowInside=->,ArrowInsidePos=0.5](7.,-1.71)(9.,-1.71)
\psline[linewidth=2.pt,ArrowInside=->,ArrowInsidePos=0.5](9.,-1.71)(8.,0.)
\psline[linewidth=2.pt,ArrowInside=->,ArrowInsidePos=0.5](10.,0.)(9.,-1.71)
\psline[linewidth=2.pt,ArrowInside=->,ArrowInsidePos=0.5](9.,-1.71)(11.,-1.71)
\psline[linewidth=2.pt,ArrowInside=->,ArrowInsidePos=0.5](11.,-1.71)(10.,0.)
\begin{scriptsize}
\psdots[dotstyle=*,linecolor=black](1.,1.71)
\psdots[dotstyle=*,linecolor=black](3.,1.712)
\psdots[dotstyle=*,linecolor=black](5.,1.71)
\psdots[dotstyle=*,linecolor=black](7.,1.71)
\psdots[dotstyle=*,linecolor=black](9.,1.71)
\psdots[dotstyle=*,linecolor=black](3.,-1.71)
\psdots[dotstyle=*,linecolor=black](5.,-1.71)
\psdots[dotstyle=*,linecolor=black](7.,-1.71)
\psdots[dotstyle=*,linecolor=black](9.,-1.71)
\psdots[dotstyle=*,linecolor=black](11.,-1.71)
\psdots[dotstyle=*,linecolor=black](8.,0.)
\psdots[dotstyle=*,linecolor=black](2.,0.)
\psdots[dotstyle=*,linecolor=black](4.,0.)
\psdots[dotstyle=*,linecolor=black](6.,0.)
\psdots[dotstyle=*,linecolor=black](10.,0.)
\rput[bl](1.1,0.82){$x_0$}
\rput[bl](1.98,1.8){$x_3$}
\rput[bl](2.7,0.82){$x_4$}
\rput[bl](2.0,-0.88){$x_2$}
\rput[bl](2.98,-0.3){$x_5$}
\rput[bl](3.6,0.82){$x_1$}
\rput[bl](4.,1.8){$x_5$}
\rput[bl](4.7,0.82){$x_0$}
\rput[bl](4.98,-0.3){$x_3$}
\rput[bl](5.6,0.82){$x_4$}
\rput[bl](5.98,1.8){$x_7$}
\rput[bl](6.7,0.82){$x_0$}
\rput[bl](6.98,-0.3){$x_1$}%
\rput[bl](7.6,0.82){$x_5$}
\rput[bl](7.98,1.8){$x_1$}
\rput[bl](8.7,0.82){$x_4$}
\rput[bl](8.98,-0.3){$x_7$}
\rput[bl](9.6,0.82){$x_0$}
\rput[bl](3.7,-0.88){$x_6$}
\rput[bl](3.98,-2.02){$x_1$}
\rput[bl](4.6,-0.88){$x_2$}
\rput[bl](5.7,-0.88){$x_7$}
\rput[bl](5.98,-2.02){$x_3$}
\rput[bl](6.7,-0.88){$x_6$}
\rput[bl](7.7,-0.88){$x_2$}
\rput[bl](7.98,-2.02){$x_5$}
\rput[bl](8.6,-0.88){$x_6$}
\rput[bl](9.7,-0.88){$x_3$}
\rput[bl](9.98,-2.02){$x_7$}
\rput[bl](10.6,-0.88){$x_2$}
\end{scriptsize}
\end{pspicture*}
\end{center}
  \caption{A van Kampen diagram over the presentation $P_8(x_0x_{1}x_5)$ with boundary label $(x_2x_0)(x_3x_5x_7x_1)(x_2x_0)^{-1}(x_1x_3x_5x_7)^{-1}$.\label{fig:positiveflat}}
\end{figure}

For later reference (in Section~\ref{sec:nonpositivecase}) we note that the relabelling of generators $y_0=x_0, y_1=x_7^{-1}, y_2=x_2, y_3=x_1^{-1}, y_4=x_4, y_5=x_3^{-1}, y_6=x_6, y_7=x_5^{-1}$ shows that $G_8(x_0x_1x_5)\cong G_8(y_0y_4y_1^{-1})$, and so we have:

\begin{corollary}\label{cor:G8nonposnothyp}
The group $H(8,4)=G_8(x_0x_4x_1^{-1})$ is not hyperbolic.
\end{corollary}

\begin{remark}\label{rk:Z+Zin(n,n/2+2)}
\em The van Kampen diagram arising in the proof of Lemma~\ref{lem:n=8nonhyp}, and later the one in the proof of Lemma~\ref{lem:H(n,n/2+2)nothyp}, provides a pair of commuting elements whose axes in the geometric realisation $\Delta$ meet at an angle $2\pi/3$. It follows that the groups considered in these results contain a free abelian subgroup of rank~2 (see, for example, \cite[page~446]{WiseApproximating}).\em
\end{remark}

In Corollaries~\ref{cor:n=21nonhyp},\ref{cor:n=24nonhyp}  we use Lemmas~\ref{lem:n=7nonhyp},\ref{lem:n=8nonhyp} respectively to prove that the groups in cases~(c),(d) of Lemma~\ref{lem:combinationsofEFG} are not hyperbolic. To do this we first recall the shift extension of a cyclically presented group. The shift automorphism $\theta$ of a cyclically presented group $G_n(w)$ results in a $\Z_n$-action on $G_n(w)$ that determines the \em shift extension \em $E_n(w) = G_n(w) \rtimes_\theta \Z_n$, which admits a two-generator two-relator presentation of the form
\[E_n(W)=\pres{x,t}{t^n,W(x,t)} \]
where $W=W(x, t)$ is obtained by rewriting $w$ in terms of the substitutions $x_i=t^ixt^{-i}$ (see, for example, \cite[Theorem~4]{JWW}). Thus there is a retraction $\nu^0: E_n(W) \rightarrow \Z_n$ given by $\nu^0(t)=t$, $\nu^0(x)=t^0=1$ with kernel $G_n(w)$. Moreover, as shown in~\cite[Section~2]{BogleyShift}, for certain values of $f$ ($0\leq f<n$) there may be further retractions $\nu^f$. Specifically, by~\cite[Theorem~2.3]{BogleyShift} the kernel of a retraction $\nu^f:E_n(W) \rightarrow \Z_n$ given by $\nu^f(t)=t$, $\nu^f(x)=t^f$ is cyclically presented, generated by the elements $y_i=t^ixt^{-(i+f)}$ ($0\leq i<n$). Since (non-elementary) hyperbolicity is preserved under taking finite index subgroups and finite extensions, the group $E_n(W)$ is (non-elementary) hyperbolic if and only if the kernel of any, and hence all, of its retractions $\nu^f$ is (non-elementary) hyperbolic.

In the case $w=x_0x_kx_{l}$ we have
\begin{alignat*}{1}
E_n(W)=G_n(w) \rtimes_\theta \pres{t}{t^n}=\pres{x,t}{t^n, xt^{k}xt^{l-k}xt^{-l}}\label{eq:positiveextension}
\end{alignat*}
which admits a retraction $\nu^f:E_n\rightarrow \pres{t}{t^n}$ given by $\nu^f(t)=t, \nu^f(x)=t^f$ if and only if $3f\equiv 0$~mod~$n$; the kernel of such a retraction is the cyclically presented group $G_n(x_0x_{f+k}x_{2f+l})$ (see~\cite[page~158]{BogleyShift}).

\begin{corollary}\label{cor:n=21nonhyp}
The group $G_{21}(x_0x_1x_5)$ is not hyperbolic.
\end{corollary}

\begin{proof}
The free product of three copies of $G_7(x_0x_1x_5)$ is the cyclically presented group $G_{21}(x_0x_3x_{15})$  with shift extension $E=\pres{x,t}{t^{21}, xt^3xt^{12}xt^{-15}}$. The kernel of the retraction $\nu^7:E\rightarrow \Z_n=\pres{t}{t^{21}}$ given by $\nu^7(t)=t$, $\nu^7(x)=t^7$ is the group $G_{21}(x_0x_{10}x_8)$ which, by~\cite[Lemma~2.1(iv),(v)]{EdjvetWilliams}, is isomorphic to $G_{21}(x_0x_1x_5)$. Since $G_7(x_0x_1x_5)$ is not hyperbolic, neither is $G_{21}(x_0x_3x_{15})$, nor $E$, and hence, nor is $G_{21}(x_0x_1x_5)$.
\end{proof}

\begin{corollary}\label{cor:n=24nonhyp}
The group $G_{24}(x_0x_1x_5)$ is not hyperbolic.
\end{corollary}

\begin{proof}
The free product of three copies of $G_8(x_0x_1x_5)$ is the cyclically presented group $G_{24}(x_0x_3x_{15})$ with shift extension $E=\pres{x,t}{t^{24}, xt^3xt^{12}xt^{-15}}$. The kernel of the retraction $\nu^8:E\rightarrow \Z_{24}=\pres{t}{t^{24}}$ given by $\nu^8(t)=t$, $\nu^8(x)=t^8$ is the group $G_{24}(x_0x_{11}x_7)$ which, by~\cite[Lemma~2.1(v),(ii)]{EdjvetWilliams}, is isomorphic to $G_{24}(x_0x_1x_5)$. Since $G_8(x_0x_1x_5)$ is not hyperbolic, neither is $G_{24}(x_0x_3x_{15})$, nor $E$, and hence, nor is $G_{24}(x_0x_1x_5)$.
\end{proof}

\subsection{Analysis of van Kampen diagrams over $P_n(x_0x_kx_l)$}\label{sec:curvaturepositive}

In this section we show that if the cyclic presentation $P=P_n(x_0x_kx_l)$ is $T(6)$ and at most one of the congruences $(E.j),(F1.j)$, $(F2.j)$ hold then $G=G_n(x_0x_kx_l)$ is hyperbolic. Following the method of proof of~\cite[Theorem~13]{HowieWilliams} we show that $G$ has a linear isoperimetric function~\cite[Theorem~3.1]{GerstenIsoperimetric}. That is, we show that there is a linear function $f:\mathbb{N}\rightarrow \mathbb{N}$ such that for all $N\in \mathbb{N}$ and all freely reduced words $W\in F_n$ with length at most $N$ that represent the identity of $G$ we have $\mathrm{Area}(W)\leq f(N)$, where $\mathrm{Area} (W)$ denotes the minimum number of faces in a reduced van Kampen diagram over $P$ with boundary label $W$. Without loss of generality we may assume that the boundary of such a van Kampen diagram $D$ is a simple closed curve. Note that each face in $D$ is a triangle, as shown in Figure~\ref{fig:positiveface}, where the corner labels $X,Y,Z$ correspond to the edge types of the star graph of $P$. We say that a vertex of $D$ is a \em boundary vertex \em if it lies on $\partial D$, and is an \em interior vertex \em otherwise. In order to obtain a linear isoperimetric function (in Lemma~\ref{lem:curvaturepositivecase}) we first carefully analyse degrees of vertices within $D$.

\begin{figure}
\begin{center}

\psset{xunit=1cm,yunit=1cm,algebraic=true,dimen=middle,dotstyle=o,dotsize=7pt 0,linewidth=1.6pt,arrowsize=3pt 2,arrowinset=0.25}
\begin{pspicture*}(9.5,13.5)(15,17)
\psline[linewidth=2pt,ArrowInside=->,ArrowInsidePos=0.5](12,16)(10,14)
\psline[linewidth=2pt,ArrowInside=->,ArrowInsidePos=0.5](10,14)(14,14)
\psline[linewidth=2pt,ArrowInside=->,ArrowInsidePos=0.5](14,14)(12,16)
\rput[tl](11.8,15.7){$X$}
\rput[tl](10.6,14.4){$Y$}
\rput[tl](13.2,14.4){$Z$}
\rput[tl](10.2,15.4){$x_{i+k}$}
\rput[tl](13.0,15.4){$x_i$}
\rput[tl](11.5,13.8){$x_{i+l}$}
\begin{scriptsize}
\psdots[dotstyle=*,linecolor=black](12,16)
\psdots[dotstyle=*,linecolor=black](10,14)
\psdots[dotstyle=*,linecolor=black](14,14)
\end{scriptsize}
\end{pspicture*}
\end{center}
  \caption{A typical face in a van Kampen diagram over the presentation $P_n(x_0x_kx_l)$.\label{fig:positiveface}}
\end{figure}
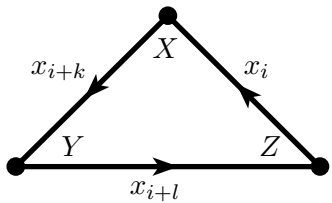

\begin{lemma}\label{lem:twolabelled(XYZ)^2}
Let $\Delta$ be an interior face of $D$ in which two of the vertices have label $(XYZ)^2$. Then the label of the third vertex contains a subword of the form $aba$, where $b$ is the label of the corner of $\Delta$ at this vertex, and $a,b \in \{X,Y,Z\}$, $a\neq b$.
\end{lemma}

\begin{proof}
Without loss of generality we may assume that the edges of $\Delta$ are oriented in an anticlockwise manner. We name its vertices $v_1,v_2,v_3$, read in an anticlockwise manner and suppose $v_1,v_2$ are labelled $(XYZ)^2$. If the corner of $\Delta$ at $v_1$ has label $X$ (resp. $Y$, resp. $Z$) then the corner of $\Delta$ at $v_2$ has label $Y$ (resp. $Z$, resp. $X$), in which case the label of $v_3$ has a subword $YZY$ or $XZX$ (resp. $ZXZ$ or $YXY$, resp. $XYX$ or $ZYZ$), as shown in Figure~\ref{fig:ICFig2}.
\end{proof}

\begin{figure}

\begin{center}

\begin{tabular}{c@{}c}
\psset{xunit=1cm,yunit=1cm,algebraic=true,dimen=middle,dotstyle=o,dotsize=7pt 0,linewidth=1.6pt,arrowsize=3pt 2,arrowinset=0.25}
\begin{pspicture*}(9,13)(15,17)
\psline[linewidth=2pt,ArrowInside=->,ArrowInsidePos=0.5](12,16)(10,14)
\psline[linewidth=2pt,ArrowInside=->,ArrowInsidePos=0.5](10,14)(14,14)
\psline[linewidth=2pt,ArrowInside=->,ArrowInsidePos=0.5](14,14)(12,16)
\rput[tl](11.82,15.7){$X$}
\rput[tl](11.2,15.8){$Y$}
\rput[tl](12.6,15.8){$Z$}
\rput[tl](9.9,14.6){$Z$}
\rput[tl](10.5,13.8){$X$}
\rput[tl](10.5,14.4){$Y$}
\rput[tl](13.4,13.8){$Y$}
\rput[tl](13.2,14.4){$Z$}
\rput[tl](13.8,14.6){$Y$}
\psline[linewidth=2pt](9.74,14.47)(10.22,13.59)
\psline[linewidth=2pt](14.34,14.47)(13.7,13.51)
\psline[linewidth=2pt](11.42,16.01)(12.56,15.99)
\rput[tl](11.78,16.4){$v_1$}
\rput[tl](9.4,13.9){$v_2$}
\rput[tl](14.18,13.9){$v_3$}
\begin{scriptsize}
\psdots[dotstyle=*,linecolor=black](12,16)
\psdots[dotstyle=*,linecolor=black](10,14)
\psdots[dotstyle=*,linecolor=black](14,14)
\end{scriptsize}
\end{pspicture*}

&

\psset{xunit=1cm,yunit=1cm,algebraic=true,dimen=middle,dotstyle=o,dotsize=7pt 0,linewidth=1.6pt,arrowsize=3pt 2,arrowinset=0.25}
\begin{pspicture*}(9,13)(15,17)
\psline[linewidth=2pt,ArrowInside=->,ArrowInsidePos=0.5](12,16)(10,14)
\psline[linewidth=2pt,ArrowInside=->,ArrowInsidePos=0.5](10,14)(14,14)
\psline[linewidth=2pt,ArrowInside=->,ArrowInsidePos=0.5](14,14)(12,16)

\rput[tl](11.82,15.7){$X$}
\rput[tl](11.2,15.8){$Z$}
\rput[tl](12.6,15.8){$Y$}
\rput[tl](10.5,14.4){$Y$}
\rput[tl](9.9,14.6){$X$}
\rput[tl](10.5,13.8){$Z$}
\rput[tl](13.4,13.8){$X$}
\rput[tl](13.8,14.6){$X$}
\rput[tl](13.2,14.4){$Z$}
\psline[linewidth=2pt](9.74,14.47)(10.22,13.59)
\psline[linewidth=2pt](14.34,14.47)(13.7,13.51)
\psline[linewidth=2pt](11.42,16.01)(12.56,15.99)
\rput[tl](11.78,16.4){$v_1$}
\rput[tl](9.4,13.9){$v_2$}
\rput[tl](14.18,13.9){$v_3$}
\begin{scriptsize}
\psdots[dotstyle=*,linecolor=black](12,16)
\psdots[dotstyle=*,linecolor=black](10,14)
\psdots[dotstyle=*,linecolor=black](14,14)
\end{scriptsize}
\end{pspicture*}

\end{tabular}
\end{center}
  \caption{Possible configurations when two vertices have label $(XYZ)^2$.\label{fig:ICFig2}}
\end{figure}
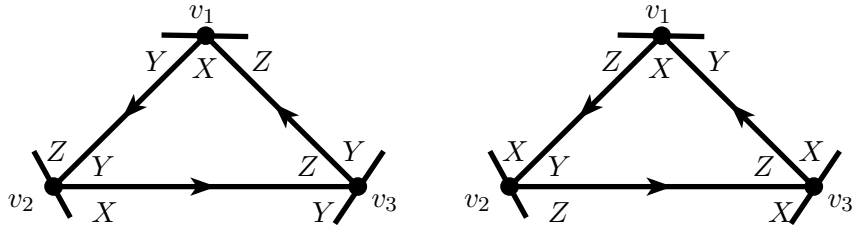

\begin{figure}
\begin{center}

\psset{xunit=1cm,yunit=1cm,algebraic=true,dimen=middle,dotstyle=o,dotsize=7pt 0,linewidth=1.6pt,arrowsize=3pt 2,arrowinset=0.25}
\begin{pspicture*}(4,7.8)(11,13)

\rput[tl](6.7,10.4){$X$}
\rput[tl](7.42,10.6){$Y$}
\rput[tl](7.9,10.4){$Z$}
\rput[tl](6.8,9.8){$Z$}
\rput[tl](7.9,9.8){$X$}
\rput[tl](7.42,9.6){$Y$}

\rput[tl](9.4,10.4){$Y$}
\rput[tl](9.4,9.8){$Z$}

\rput[tl](5.3,10.4){$Y$}
\rput[tl](5.3,9.8){$X$}

\rput[tl](8.3,8.4){$Z$}
\rput[tl](8.8,8.6){$Y$}

\rput[tl](6.4,11.9){$X$}
\rput[tl](5.9,11.6){$Z$}

\rput[tl](5.94,8.5){$Y$}
\rput[tl](6.4,8.4){$X$}

\rput[tl](8.8,11.6){$X$}
\rput[tl](8.4,11.85){$Z$}

\psline[linewidth=2pt,ArrowInside=->,ArrowInsidePos=0.5](5,10)(6,8)
\psline[linewidth=2pt,ArrowInside=->,ArrowInsidePos=0.5](9,8)(6,8)
\psline[linewidth=2pt,ArrowInside=->,ArrowInsidePos=0.5](9,8)(10,10)
\psline[linewidth=2pt,ArrowInside=->,ArrowInsidePos=0.5](9,12)(10,10)
\psline[linewidth=2pt,ArrowInside=->,ArrowInsidePos=0.5](9,12)(6,12)
\psline[linewidth=2pt,ArrowInside=->,ArrowInsidePos=0.5](5,10)(6,12)
\psline[linewidth=2pt,ArrowInside=->,ArrowInsidePos=0.5](6,12)(7.5,10)
\psline[linewidth=2pt,ArrowInside=->,ArrowInsidePos=0.5](10,10)(7.5,10)
\psline[linewidth=2pt,ArrowInside=->,ArrowInsidePos=0.5](6,8)(7.5,10)
\psline[linewidth=2pt,ArrowInside=->,ArrowInsidePos=0.5](7.5,10)(5,10)
\psline[linewidth=2pt,ArrowInside=->,ArrowInsidePos=0.5](7.5,10)(9,12)
\psline[linewidth=2pt,ArrowInside=->,ArrowInsidePos=0.5](7.5,10)(9,8)

\begin{scriptsize}
\psdots[dotstyle=*,linecolor=black](6,12)
\psdots[dotstyle=*,linecolor=black](5,10)
\psdots[dotstyle=*,linecolor=black](10,10)
\psdots[dotstyle=*,linecolor=black](9,8)
\psdots[dotstyle=*,linecolor=black](6,8)
\psdots[dotstyle=*,linecolor=black](7.5,10)
\psdots[dotstyle=*,linecolor=black](9,12)
\end{scriptsize}
\end{pspicture*}
\end{center}
  \caption{Neighbourhood of an interior vertex labelled $(XYZ)^2$.\label{fig:ICFig3}}
\end{figure}
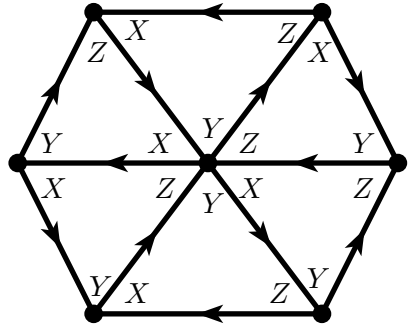

\begin{lemma}\label{lem:INT(XYZ)^2twoneighbours}
If an interior vertex $v$ of $D$ of degree 6 has label $(XYZ)^2$ then two adjacent neighbours of $v$ have $XY$ as a cyclic subword of their labels, two adjacent neighbours have $XZ$ as a cyclic subword of their labels, and two adjacent neighbours have $YZ$ as a cyclic subword of their labels.
\end{lemma}

\begin{proof}
If the label of $v$ is $(XYZ)^2$ oriented clockwise, then the neighbourhood of $v$ is as given in Figure~\ref{fig:ICFig3}, from which the conclusion can be observed. A similar figure deals with the case when the label of $v$ is $(XYZ)^2$ oriented anticlockwise.
\end{proof}

\begin{figure}[ht]
\begin{center}

\psset{xunit=1cm,yunit=1cm,algebraic=true,dimen=middle,dotstyle=o,dotsize=7pt 0,linewidth=1.6pt,arrowsize=3pt 2,arrowinset=0.25}
\begin{pspicture*}(2,6.0)(13,14.8)
%
\rput[tl](11.1,8.5){$u_1$}
\rput[tl](10.4,8.6){$Z$}
\rput[tl](7.4,6.3){$u_2$}
\rput[tl](7.4,7.2){$Z$}
\rput[tl](3.3,8.5){$u_3$}
\rput[tl](4.3,8.6){$Z$}
\rput[tl](3.3,12.0){$u_4$}
\rput[tl](4.3,11.7){$Z$}
\rput[tl](7.3,13.9){$u_5$}
\rput[tl](7.3,13.1){$Z$}
\rput[tl](11.3,12.0){$u_6$}
\rput[tl](10.3,11.7){$Z$}
%
\rput[tl](6.8,10.4){$X$}
\rput[tl](7.4,10.5){$Y$}
\rput[tl](7.9,10.4){$X$}
\rput[tl](6.8,9.9){$Y$}
\rput[tl](7.3,9.7){$X$}
\rput[tl](7.9,9.9){$Y$}

\rput[tl](9.3,9.8){$X$}
\rput[tl](9.9,10.6){$Y$}
\rput[tl](9.9,9.5){$Y$}
\rput[tl](9.4,10.4){$Z$}

\rput[tl](8.3,8.4){$Y$}
\rput[tl](8.2,7.9){$X$}
\rput[tl](9.3,8.5){$X$}
\rput[tl](8.8,8.7){$Z$}

\rput[tl](8.8,11.6){$Y$}
\rput[tl](8.1,12.4){$X$}
\rput[tl](9.3,11.8){$X$}
\rput[tl](8.4,11.9){$Z$}
\rput[tl](6.4,11.9){$X$}
\rput[tl](5.4,11.8){$Y$}
\rput[tl](5.86,11.7){$Z$}
\rput[tl](6.6,12.4){$Y$}
\rput[tl](6.4,8.4){$Z$}
\rput[tl](5.86,8.8){$X$}
\rput[tl](5.4,8.5){$Y$}
\rput[tl](6.6,7.9){$Y$}


\rput[tl](4.84,10.7){$X$}
\rput[tl](4.78,9.6){$X$}
\rput[tl](5.26,9.9){$Z$}
\rput[tl](5.28,10.4){$Y$}

\psline[linewidth=2pt,ArrowInside=->,ArrowInsidePos](7.5,6.65)(9,8)
\psline[linewidth=2pt,ArrowInside=->,ArrowInsidePos](6,8)(7.5,6.65)
\psline[linewidth=2pt,ArrowInside=->,ArrowInsidePos](9,8)(6,8)
\psline[linewidth=2pt,ArrowInside=->,ArrowInsidePos](6,8)(7.5,10)
\psline[linewidth=2pt,ArrowInside=->,ArrowInsidePos](7.5,10)(9,8)
\psline[linewidth=2pt,ArrowInside=->,ArrowInsidePos](7.5,10)(9,12)
\psline[linewidth=2pt,ArrowInside=->,ArrowInsidePos](7.5,10)(5,10)
\psline[linewidth=2pt,ArrowInside=->,ArrowInsidePos](10,10)(7.5,10)
\psline[linewidth=2pt,ArrowInside=->,ArrowInsidePos](6,12)(7.5,10)
\psline[linewidth=2pt,ArrowInside=->,ArrowInsidePos](6,8)(7.5,10)
\psline[linewidth=2pt,ArrowInside=->,ArrowInsidePos](5,10)(6,8)
\psline[linewidth=2pt,ArrowInside=->,ArrowInsidePos](5,10)(6,12)
\psline[linewidth=2pt,ArrowInside=->,ArrowInsidePos](9,12)(10,10)
\psline[linewidth=2pt,ArrowInside=->,ArrowInsidePos](10,10)(11,8.25)
\psline[linewidth=2pt,ArrowInside=->,ArrowInsidePos](6,8)(4,8.25)
\psline[linewidth=2pt,ArrowInside=->,ArrowInsidePos](4,8.25)(5,10)
\psline[linewidth=2pt,ArrowInside=->,ArrowInsidePos](5,10)(6,8)
\psline[linewidth=2pt,ArrowInside=->,ArrowInsidePos](4,11.75)(5,10)
\psline[linewidth=2pt,ArrowInside=->,ArrowInsidePos](5,10)(6,12)
\psline[linewidth=2pt,ArrowInside=->,ArrowInsidePos](6,12)(4,11.75)
\psline[linewidth=2pt,ArrowInside=->,ArrowInsidePos](6,12)(7.5,13.35)
\psline[linewidth=2pt,ArrowInside=->,ArrowInsidePos](7.5,13.35)(9,12)
\psline[linewidth=2pt,ArrowInside=->,ArrowInsidePos](9,12)(6,12)
\psline[linewidth=2pt,ArrowInside=->,ArrowInsidePos](11,11.75)(9,12)
\psline[linewidth=2pt,ArrowInside=->,ArrowInsidePos](10,10)(11,11.75)
\psline[linewidth=2pt,ArrowInside=->,ArrowInsidePos](9,8)(10,10)
\psline[linewidth=2pt,ArrowInside=->,ArrowInsidePos](11,8.25)(9,8)

\begin{scriptsize}
\psdots[dotstyle=*,linecolor=black](5,10)
\psdots[dotstyle=*,linecolor=black](7.5,10)
\psdots[dotstyle=*,linecolor=black](10,10)

\psdots[dotstyle=*,linecolor=black](6,8)
\psdots[dotstyle=*,linecolor=black](9,8)

\psdots[dotstyle=*,linecolor=black](6,12)
\psdots[dotstyle=*,linecolor=black](9,12)

\psdots[dotstyle=*,linecolor=black](4,11.75)
\psdots[dotstyle=*,linecolor=black](11,11.75)

\psdots[dotstyle=*,linecolor=black](7.5,13.35)
\psdots[dotstyle=*,linecolor=black](7.5,6.65)

\psdots[dotstyle=*,linecolor=black](4,8.25)
\psdots[dotstyle=*,linecolor=black](11,8.25)

\end{scriptsize}
\end{pspicture*}

\end{center}
  \caption{Neighbourhood of an interior vertex labelled $(E.0)$ and no boundary neighbours.\label{fig:ICFig4}}
\end{figure}
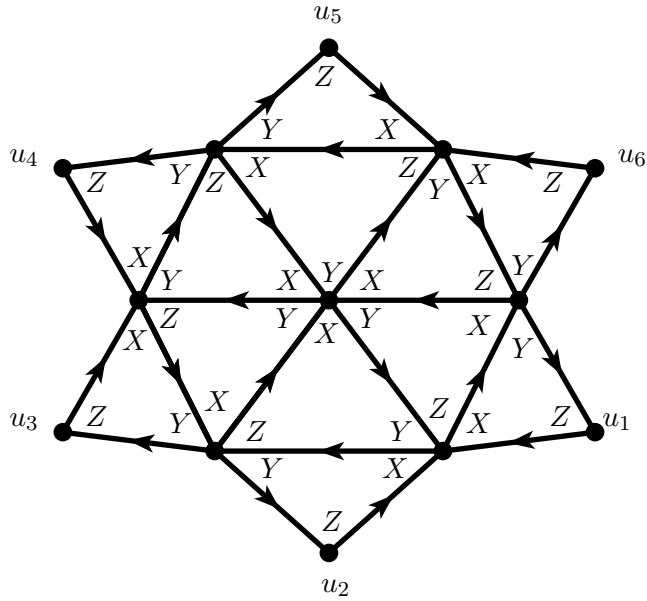

\begin{lemma}\label{lem:label(E.j)}
Suppose that all interior vertices of $D$ have degree at least 6 and all labels of interior vertices of degree 6 are either $(XYZ)^2$ or label $(E.j)$ for precisely one $j\in \{0,1,2\}$.
If $v$ is an interior vertex of degree 6 in $D$ with label $(E.j)$ and where all the neighbours of $v$ are interior vertices of degree 6 then every neighbour of $v$ has two neighbours which are either boundary vertices or have degree at least~8.
\end{lemma}

\begin{proof}
Consider first the case $(E.0)$, that is, a vertex label $(XY)^3$. As shown in Figure~\ref{fig:ICFig4} all the neighbours of $v$ must have label $(XYZ)^2$. Then each of the vertices $u_1,\ldots ,u_6$ has a corner labelled $Z$. If a vertex $u_i$ ($1\leq i\leq 6$) is interior then if it is of degree 6, its label is not $(XYZ)^2$, by Lemma~\ref{lem:twolabelled(XYZ)^2}, and so it must be $(XY)^3$, a contradiction. Therefore $u_i$ is either interior of degree at least 8, or a boundary vertex, as required. The cases $(E.1),(E.2)$ are dealt with by replacing $X$ by $Y$, $Y$ by $Z$, and $Z$ by $X$, as explained in Section~\ref{sec:positiveshortcycles}.
\end{proof}

\begin{figure}[ht]
\begin{center}

\psset{xunit=1cm,yunit=1cm,algebraic=true,dimen=middle,dotstyle=o,dotsize=7pt 0,linewidth=1.6pt,arrowsize=3pt 2,arrowinset=0.25}
\begin{pspicture*}(4,7.5)(11,13)
\rput[tl](6.7,10.4){$X$}
\rput[tl](7.4,10.6){$Y$}
\rput[tl](7.9,9.8){$Z$}
\rput[tl](7.3,9.7){$X$}
\rput[tl](6.8,9.8){$Y$}
\rput[tl](7.8,10.4){$X$}

\rput[tl](8.3,8.4){$Y$}
\rput[tl](5.9,8.6){$X$}
\rput[tl](6.4,8.4){$Z$}
\rput[tl](6.5,11.8){$X$}
\rput[tl](5.9,11.6){$Z$}
\rput[tl](5.3,9.9){$Z$}
\rput[tl](8.8,11.6){$Y$}
\rput[tl](8.7,8.7){$X$}
\rput[tl](8.3,11.8){$Z$}
%
\rput[tl](9.4,10.4){$Z$}
\rput[tl](9.4,9.8){$Y$}

\rput[tl](5.3,10.4){$Y$}


\psline[linewidth=2pt,ArrowInside=->,ArrowInsidePos=0.5](5,10)(6,8)
\psline[linewidth=2pt,ArrowInside=->,ArrowInsidePos=0.5](9,8)(6,8)
\psline[linewidth=2pt,ArrowInside=->,ArrowInsidePos=0.5](9,8)(10,10)
\psline[linewidth=2pt,ArrowInside=->,ArrowInsidePos=0.5](9,12)(10,10)
\psline[linewidth=2pt,ArrowInside=->,ArrowInsidePos=0.5](9,12)(6,12)
\psline[linewidth=2pt,ArrowInside=->,ArrowInsidePos=0.5](5,10)(6,12)
\psline[linewidth=2pt,ArrowInside=->,ArrowInsidePos=0.5](6,12)(7.5,10)
\psline[linewidth=2pt,ArrowInside=->,ArrowInsidePos=0.5](10,10)(7.5,10)
\psline[linewidth=2pt,ArrowInside=->,ArrowInsidePos=0.5](6,8)(7.5,10)
\psline[linewidth=2pt,ArrowInside=->,ArrowInsidePos=0.5](7.5,10)(5,10)
\psline[linewidth=2pt,ArrowInside=->,ArrowInsidePos=0.5](7.5,10)(9,12)
\psline[linewidth=2pt,ArrowInside=->,ArrowInsidePos=0.5](7.5,10)(9,8)

\begin{scriptsize}
\psdots[dotsize=7pt,dotstyle=*,linecolor=black](6,12)
\psdots[dotsize=7pt,dotstyle=*,linecolor=black](5,10)
\psdots[dotsize=7pt,dotstyle=*,linecolor=black](10,10)
\psdots[dotsize=7pt,dotstyle=*,linecolor=black](9,8)
\psdots[dotsize=7pt,dotstyle=*,linecolor=black](6,8)
\psdots[dotsize=7pt,dotstyle=*,linecolor=black](7.5,10)
\psdots[dotsize=7pt,dotstyle=*,linecolor=black](9,12)
\end{scriptsize}
\end{pspicture*}

\end{center}
  \caption{Neighbourhood of an interior vertex labelled $(F1.0)$.\label{fig:ICFig5}}
\end{figure}
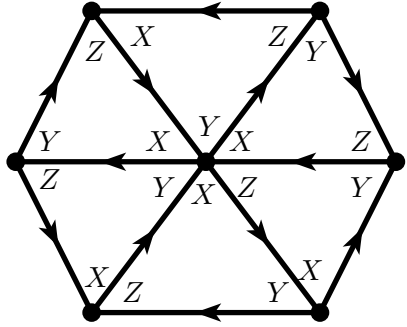

\begin{lemma}\label{lem:label(F1F2.j)}
Suppose that all interior vertices of $D$ have degree at least 6 and all labels of interior vertices of degree 6 are either
$(XYZ)^2$ or $(F1.j)$ (resp. $(F2.j)$) for precisely one $j\in \{0,1,2\}$.
If $v$ is an interior vertex of degree 6 in $D$ with label $(F1.j)$ (resp. $(F2.j)$) and where all the neighbours of $v$ are interior vertices then $v$ has a neighbour of degree at least 8.
\end{lemma}

\begin{proof}
Consider the case $(F1.0)$, that is, $v$ has label $(XY)^2XZ$ and suppose that all neighbours of $v$ have degree~6. Then Figure~\ref{fig:ICFig5} shows one of the two possible labellings of neighbours that can occur. Since two adjacent neighbours have $YZ$ as a cyclic subword of their label, these must each be labelled $(XYZ)^2$, but this is impossible by Lemma~\ref{lem:twolabelled(XYZ)^2}; therefore $v$ has a neighbour of degree at least~8. The same conclusion can be obtained if the second possible labelling of neighbours occurs. The cases $(F1.1),(F1.2)$ are dealt with by replacing $X$ by $Y$, $Y$ by $Z$, and $Z$ by $X$. The cases $(F2.j)$ are obtained from the cases $(F1.j)$ by interchanging the roles of $X$ and $Y$, as described in Section~\ref{sec:positiveshortcycles}.
\end{proof}

We are now in a position to be able to establish the existence of a suitable isoperimetric function.

\begin{lemma}\label{lem:curvaturepositivecase}
Let $n\geq 2$ and suppose that none of the congruences $(B.j),(C.j),(D.j)$ hold ($0\leq j\leq 2)$ and that at most one congruence $(E.j)$, $(F1.j)$, $(F2.j)$ holds ($j\in \{0,1,2\}$). Then $G_n(x_0x_kx_{l})$ has a linear isoperimetric function.
\end{lemma}

\begin{proof}
As at the beginning of this section, let $N\in \mathbb{N}$ and let $W$ be a freely reduced word in the free group $F_n$ (with generators $x_0,\ldots ,x_{n-1})$ of length at most $N$ that represents the identity of $G$ and let $D$ be a reduced van Kampen diagram whose boundary $\partial D$ is a simple closed curve with label $W$. We let $I$ denote the set of interior vertices of $D$, $B$ denote the set of boundary vertices of $D$, and $F$ denote the set of faces of $D$. Then $\mathrm{Area}(W)\leq |F|$. Writing $\pi$ to denote $180$, we define the curvature of a face $f$ by $\kappa(f)=-\pi + (\mathrm{sum~of~angles~in}~f)$, the curvature of an interior vertex $v$ by $\kappa (v)=2\pi - (\mathrm{sum~of~angles~at}~v)$, and the curvature of a boundary vertex $\hat{v}$ by \linebreak $\kappa(\hat{v})=\pi - (\mathrm{sum~of~angles~at}~\hat{v})$. It follows from the Gauss-Bonnet Theorem that
\begin{equation}
  \sum_{v\in I} \kappa(v) + \sum_{\hat{v}\in B} \kappa (\hat{v}) +\sum_{f\in F} \kappa(f)=2\pi \label{eq:GaussBonnet}
\end{equation}
(see \cite[Section~4]{MccammondWise} and the references therein).

Since none of the congruences $(B.j),(C.j),(D.j)$ hold ($0\leq j\leq 2)$ every interior vertex of $D$ is of degree at least 6, and since at most one congruence $(E.j)$, $(F1.j)$, $(F2.j)$ holds, the label of an interior vertex of degree~6 is either $(XYZ)^2$ or it is the label corresponding to that congruence, given in Table~\ref{tab:BCDEFG}.

We assign angles to the corners of the faces in $D$ as follows. If $v$ is a boundary vertex or an interior vertex of degree at least 8 then assign angle 47 to every corner at $v$. Assume now that $v$ is an interior vertex of degree~6 and consider a face $f$ with vertices $v$ and $u,w$: if $u,w$ are interior of degree 6 then assign angle 59 to the corner of $f$ at $v$; otherwise assign 66 to the corner of $f$ at $v$.

Then, if a face $f$ contains a boundary vertex then $\kappa(f)\leq -\pi +(47+66+66)=-1$; if a face $f$ has all its vertices interior and one vertex of degree at least $8$ then $\kappa(f)\leq -\pi +(47+66+66)=-1$; if all the vertices of a face $f$ are interior of degree 6 then $\kappa(f)\leq -\pi+(59+59+59)=-3$. Therefore $\kappa(f)\leq -1$ for all $f\in F$.

We now consider curvature of the vertices. If $v$ is an interior vertex of degree at least 8 then $\kappa(v)\leq 2\pi -8(47) =-16$. If $v$ is an interior vertex of degree 6 with a neighbour that is either on the boundary $\partial D$ or has degree at least 8, then $\kappa(v)\leq 2\pi -2(66) -4(59)=-8$.

Now suppose that $v$ is an interior vertex of degree 6 with all its neighbours interior of degree~6. Then $\kappa(v)=2\pi -6(59)=6$ and by Lemma~\ref{lem:label(F1F2.j)} the label of $v$ is either $(XYZ)^2$ or $(E.j)$ for some $j\in \{0,1,2\}$. If the label of $v$ is $(XYZ)^2$ then, since precisely one other label of degree 6 vertices is possible, Lemma~\ref{lem:INT(XYZ)^2twoneighbours} implies that $v$ must have two adjacent neighbours, each labelled $(XYZ)^2$, but this is impossible by Lemma~\ref{lem:twolabelled(XYZ)^2}.
If the label of $v$ is $(E.j)$ (for some $j\in \{0,1,2\}$) then Lemma~\ref{lem:label(E.j)} implies that every neighbour $v_i$ ($1\leq i \leq 6$) of $v$ has two neighbours which are either boundary vertices or have degree at least 8. Therefore, for each $i\in\{1,\dots,6\}$ the curvature $\kappa (v_i)\leq 2\pi - 4(59)-2(66)=-8$. In this situation, transfer curvature of $-1$ from each vertex $v_i$ to vertex $v$; the resulting curvatures are $\kappa(v_i)\leq -8+1=-7$ ($1\leq i\leq 6$) and $\kappa(v)=6-6(1)=0$. Since each vertex $v_i$ has degree 6, the maximum number of times curvature can be transferred away from $v_i$ is 6, so its curvature cannot exceed $\kappa(v)=-8+6(1)=-2$. Therefore for each interior vertex $v$ we have $\kappa(v)\leq 0$.

Now~(\ref{eq:GaussBonnet}) implies
\begin{alignat*}{1}
  2\pi &= \sum_{v\in I} \kappa(v) + \sum_{\hat{v}\in B} \kappa (\hat{v}) +\sum_{f\in F} \kappa(f)\\
  &\leq \sum_{v\in I} 0 + \sum_{\hat{v}\in B} (\pi - \mathrm{sum~of~angles~at}~\hat{v}) +\sum_{f\in F} (-1)\\
  &= |B|\pi - \sum_{\hat{v}\in B} (\mathrm{sum~of~angles~at}~\hat{v}) - |F|
\end{alignat*}
 so
\[  \sum_{\hat{v}\in B} (\mathrm{sum~of~angles~at}~\hat{v}) \leq (|B|-2)\pi - |F|.\]
On the other hand, the corner angle at any boundary vertex is 47, so the sum of angles over the boundary vertices is bounded below by $47|B|$. Therefore $47|B|\leq (|B|-2)\pi - |F|$, so $|F| \leq 133|B| -360$.

But $\mathrm{Area}(W)\leq |F|$ and $|B|\leq N$ so $\mathrm{Area}(W)\leq 133N -360$, so $f(N)=133N-360$ is a linear isoperimetric function.
\end{proof}

We now have all the ingredients to prove Theorem~\ref{thm:positive}.

\subsection{Proof of Theorem~\ref{thm:positive}}\label{sec:proofofpositive}

Suppose that $n\geq 2$, $0\leq k,l<n$, $(n,k,l)=1$ and that the cyclic presentation $P_n(x_0x_kx_l)$ satisfies T(6). If $P_n(x_0x_kx_l)$ has a freely redundant relator then $n=3$ and $G\cong \Z*\Z$ (which is non-elementary hyperbolic) so we may assume that $P_n(x_0x_kx_l)$ has no freely redundant relators. Then Lemma~\ref{lem:positiveT6} implies that none of the congruences $(B.j),(C.j),(D.j)$ ($0\leq j\leq 2)$ (of Table~\ref{tab:BCDEFG}) hold and so $n\geq 7$. If $n=7$ or $8$ then $G\cong G_n(x_0x_1x_5)$ (see, for example, \cite[Table~2]{MohamedWilliams}) so is not hyperbolic by Lemmas~\ref{lem:n=7nonhyp},\ref{lem:n=8nonhyp}. Assume then that $n>8$. If more than one of the congruences $(E.j),(F1.j)$, $(F2.j)$ ($0\leq j\leq 2)$ hold then one of the cases (c),(d),(e)  of Lemma~\ref{lem:combinationsofEFG} hold. In cases (c),(d) $G$ is not hyperbolic by Corollaries~\ref{cor:n=21nonhyp},\ref{cor:n=24nonhyp} and in case (e) $G$ is non-elementary hyperbolic, by Example~\ref{ex:n=27}. Thus we may assume that at most one of $(E.j),(F1.j)$, $(F2.j)$ ($0\leq j\leq 2$) hold, in which case Lemma~\ref{lem:curvaturepositivecase} implies that $G_n(x_0x_kx_l)$ has a linear isoperimetric function, and hence is hyperbolic.  By~\cite[Corollary~5.2]{EdjvetWilliams} $G$ contains a non-abelian free subgroup so it is non-elementary hyperbolic.

\section{The non-positive case}\label{sec:nonpositivecase}

As in~\cite{HowieWilliams}, we express our arguments in terms of parameters $A=k, B=k-m$.

Let $\Gamma$ be the star graph of the cyclic presentation $P_n(x_0x_mx_k^{-1})$. Then $\Gamma$ has vertices $x_i,x_i^{-1}$ and edges $x_i-x_{i+m}^{-1}$, $x_i-x_{i+B}$, $x_i^{-1}-x_{i+A}^{-1}$ ($0\leq i<n$), which we will refer to as edges of type $X,Y,Z$, respectively. Replacing parameter $k$ by $m-k$ corresponds to interchanging the roles of edges of types $Y$ and $Z$, and so will correspond to interchanging the roles of conditions $(*.0)$ and $(*.1)$ in Table~\ref{tab:nonposshortcycles}, and replacing the group $G_n(x_0x_mx_k^{-1})$ by the isomorphic copy $G_n(x_0x_mx_{m-k}^{-1})$. (To see that  $G_n(x_0x_mx_{k}^{-1})\cong G_n(x_0x_mx_{m-k}^{-1})$ replace generators $x_i$ by $x_i^{-1}$, invert the relators, negate the subscripts and set $j=-i-m$ to get the relators $x_jx_{j+m}x_{j+m-k}^{-1}$ of $G_n(x_0x_mx_{m-k}^{-1})$.)

As in the positive case, we are interested in cycles of length at most $6$ in $\Gamma$, so we analyze these in Section~\ref{sec:nonpositiveshortcycles}. We observe that if a particular cycle type of length~6 (which we refer to as $(\gamma+)$) occurs then $G=G_n(x_0x_mx_k^{-1})$ is isomorphic to $G_n(x_0x_{n/2+2}x_1^{-1})=H(n,n/2+2)$ which (in Section~\ref{sec:nonhyperbolicgroupsnonpositive}) we show is non-hyperbolic whenever its presentation satisfies $T(6)$. We then show that if two of the remaining cycle types of length~6 occur, then $G_n(x_0x_mx_k^{-1})$ is isomorphic to one of a few groups with low values of $n$, all but one of which turn out to be hyperbolic (the other, $G_8(x_0x_4x_1^{-1})=H(8,4)$, being non-hyperbolic). In Section~\ref{sec:curvaturenonpositive} we consider the case when exactly one cycle type of length~6 occurs and perform a detailed analysis of van Kampen diagrams over the defining presentation to prove that $G_n(x_0x_mx_k^{-1})$ has a linear isoperimetric function, and hence is hyperbolic. We then combine these results to prove Theorem~\ref{thm:nonpositive} in Section~\ref{sec:proofofnonpositive}.

\subsection{Analysis of short cycles in the star graph of $P_n(x_0x_mx_k^{-1})$}\label{sec:nonpositiveshortcycles}

Short cycles  in $\Gamma$ were analyzed in~\cite{HowieWilliams}:
\begin{lemma}[{\cite[Theorem~10]{HowieWilliams}}]\label{lem:nonposshortcycles}
Let $n\geq 2$, $0\leq m,k<n$, $m\neq k$, $k\neq 0$ and set $A=k,B=k-m$, and let $\Gamma$ be the star graph of $P_n(x_0x_mx_k^{-1})$.
\begin{itemize}
  \item[(a)] $\Gamma$ has a cycle of length less than $6$ if and only if at least one of the congruences $(\rho.j)$, $(\sigma+.j)$, $(\sigma-.j)$, $(\tau+.j)$, $(\tau-.j)$ of Table~\ref{tab:nonposshortcycles} holds, in which case a label of the cycle is the corresponding entry of the table.

  \item[(b)] $\Gamma$ has a cycle  of length 6 if and only if at least one of the congruences ($\alpha.j$),($\beta+.j$),($\beta-.j$), ($\gamma+$),($\gamma-$) of Table~\ref{tab:nonposshortcycles} holds, in which case a label of the cycle  is the corresponding entry of the table.
\end{itemize}
\end{lemma}

\begin{corollary}\label{cor:T6T7nonpositive}
Let $n\geq 2$ and suppose $(n,m,k)=1$, $0\leq m,k<n$, $m\neq k$, $k\neq 0$. Then
\begin{itemize}
  \item[(a)] $P_n(x_0x_mx_k^{-1})$ satisfies $T(6)$ if and only if none of the congruences $(\rho.j)$, $(\sigma+.j)$, $(\sigma-.j)$, $(\tau+.j)$, $(\tau-.j)$ of Table~\ref{tab:nonposshortcycles} holds;
  \item[(b)] $P_n(x_0x_mx_k^{-1})$ satisfies $T(7)$ if and only if none of the congruences $(\rho.j)$, $(\sigma+.j)$, $(\sigma-.j)$, $(\tau+.j)$, $(\tau-.j)$, ($\alpha.j$), ($\beta+.j$), ($\beta-.j$), ($\gamma+$), ($\gamma-$) of Table~\ref{tab:nonposshortcycles} holds.
\end{itemize}
\end{corollary}

\begin{table}
\begin{center}
\begin{tabular}{|cr|ll|}\hline
& $j$ & \textbf{0} & \textbf{1} \\\hline
($\rho$.$j$) & $m,k$ congruence & {$k-m\equiv \frac{n}{2}; \pm \frac{n}{3}; \pm \frac{n}{4}; \pm \frac{n}{5}, \pm \frac{2n}{5}$} & { $k\equiv \frac{n}{2}; \pm \frac{n}{3}; \pm \frac{n}{4}; \pm \frac{n}{5}, \pm \frac{2n}{5}$ }\\
& $A,B$ congruence& $B\equiv \frac{n}{2}; \pm \frac{n}{3}; \pm \frac{n}{4}; \pm \frac{n}{5}, \pm \frac{2n}{5}$ & {$A\equiv \frac{n}{2}; \pm \frac{n}{3}; \pm \frac{n}{4}; \pm \frac{n}{5}, \pm \frac{2n}{5}$}\\
& cycle type& {$Y^2$; $Y^3$; $Y^4$; $Y^5$} & {$Z^2$; $Z^3$; $Z^4$; $Z^5$}\\\hline
($\sigma+$) & $m,k$ congruence &$2k-m\equiv0$ & $2k-m\equiv 0$\\
& $A,B$ congruence&$A+B\equiv 0$ & $B+A\equiv 0$\\
& cycle type& $XYXZ$ & $XZXY$\\\hline
($\sigma-$) & $m,k$ congruence &$m\equiv0$ & $m\equiv 0$\\
& $A,B$ congruence&$A-B\equiv 0$ & $B-A\equiv 0$\\
& cycle type& $XYXZ$ & $XZXY$\\\hline
($\tau+.j$) & $m,k$ congruence &$3k-2m\equiv0$ & $3k-m\equiv 0$\\
& $A,B$ congruence&$A+2B\equiv 0$ & $B+2A\equiv 0$\\
& cycle type& $XZXYY$ & $XYXZZ$\\\hline
($\tau-.j$) & $m,k$ congruence &$2m-k\equiv0$ & $m+k\equiv 0$\\
& $A,B$ congruence&$A-2B\equiv 0$ & $B-2A\equiv 0$\\
& cycle type& $XZXYY$ & $XYXZZ$ \\\hline
($\alpha.j$) & $m,k$ congruence &{$k-m\equiv \pm \frac{n}{6}$} & $k\equiv \pm \frac{n}{6}$\\
& $A,B$ congruence&$B\equiv \pm \frac{n}{6}$ & $A\equiv \pm \frac{n}{6}$\\
& cycle type& $Y^6$ & $Z^6$\\\hline
($\beta+.j$) & $m,k$ congruence &$4k-3m\equiv0$ & $4k-m\equiv 0$\\
& $A,B$ congruence&$A+3B\equiv 0$ & $B+3A\equiv 0$\\
& cycle type& $XZXYYY$ & $XYXZZZ$\\\hline
($\beta-.j$) & $m,k$ congruence &$3m-2k\equiv0$ & $2k+m\equiv 0$\\
& $A,B$ congruence&$A-3B\equiv 0$ & $B-3A\equiv 0$\\
& cycle type& $XZXYYY$ & $XYXZZZ$\\\hline
($\gamma+$) & $m,k$ congruence & $2k-m\equiv n/2$ & $2k-m\equiv n/2$\\
& $A,B$ congruence&$A+B\equiv n/2$ & $B+A\equiv n/2$\\
& cycle type& $XZZXYY$ & $XYYXZZ$\\\hline
($\gamma-$) & $m,k$ congruence &$m\equiv n/2$ & $m\equiv n/2$\\
& $A,B$ congruence&$A-B\equiv n/2$ & $B-A\equiv n/2$\\
& cycle type& $XZZXYY$ & $XYYXZZ$\\\hline
\end{tabular}
\end{center}
\caption{Congruences (mod~$n$) corresponding to short cycles in the star graph of $P_n(x_0x_mx_k^{-1})$. Here $A=k, B=k-m$. \label{tab:nonposshortcycles}}
\end{table}

Note that the two ($\gamma+$) conditions in Table~\ref{tab:nonposshortcycles} are identical conditions and the two ($\gamma-$) conditions are identical; for this reason we do not add the `$.j$' to these conditions. We first identify the groups $G_n(x_0x_mx_k^{-1})$ in the presence of a cycle of type $(\gamma+)$ of Table~\ref{tab:nonposshortcycles}.

\begin{lemma}\label{lem:2(A-B)=0groups}
Let $n\geq 2$, $0\leq m,k<n$ and let $A=k,B=k-m$. Suppose $A+B\equiv n/2$~mod~$n$ and that $(n,m,k)=1$. Then $G_n(x_0x_mx_k^{-1})\cong G_n(x_0x_{n/2+2}x_1^{-1})=H(n,n+2)$.
\end{lemma}

\begin{proof}
The hypotheses imply $m\equiv 2k+n/2$~mod~$n$ so $1=(n,m,k)=(n,n/2+2k,k)$ which implies $(n/2,k)=1$ so either $(n,k)=1$ or ($n/2$ is odd and $(n,k)=2$). In the former case $G_n(x_0x_mx_k^{-1}) \cong G_n(x_0x_{n/2+2k}x_k^{-1})\cong G_n(x_0x_{n/2+2}x_1^{-1})$ (by~\cite[Lemma~1.3]{BardakovVesnin}); in the latter case $G_n(x_0x_mx_k^{-1}) \cong G_n(x_0x_{n/2+2k}x_k^{-1})\cong G_n(x_0x_{n/2+4}x_2^{-1})$ which is isomorphic to $G_n(x_0x_{n/2+2}x_1^{-1})$ by~\cite[Lemma~1.3]{BardakovVesnin}, \cite[Lemma~7]{W-CHR}.
\end{proof}

In Lemma~\ref{lem:H(n,n/2+2)nothyp} we will show that the groups $H(n,n/2+2)$ are not hyperbolic for any even $n\geq 8$, $n \neq 10$. We now consider the groups that arise when more than one of the remaining length~6 cycle cases hold.

\begin{lemma}\label{lem:nonposmorethanonelength6NewVersion}
Let $n\geq 2$, $0\leq m,k<n$, $m\neq k$, $k\neq 0$ where $(n,m,k)=1$, and set $A=k,B=k-m$, and let $\Gamma$ be the star graph of $P_n(x_0x_mx_k^{-1})$. Suppose that none of the congruences $(\rho.j)$, $(\sigma+.j)$, $(\sigma-.j)$, $(\tau+.j)$, $(\tau-.j)$, hold for $j\in \{0,1\}$ and that $(\gamma+)$ does not hold. If more than one of the congruences $(\alpha.j)$, $(\beta+.j)$, $(\beta-.j)$, $(\gamma-)$ holds, then $G_n(x_0x_mx_k^{-1})$ is isomorphic to one of $H(8,4)$, $H(8,6)$, $H(10,4)$, $H(18,4)$, $H(18,16)$.
\end{lemma}

\begin{proof}
If $(\alpha.0)$ and $(\alpha.1)$ hold then $A\equiv \pm n/6$ and $B\equiv \pm n/6$~mod~$n$ so either $(\sigma+)$ or $(\sigma-)$ holds, a contradiction.
If $(\alpha.0)$ and $(\gamma-)$ hold then $(\rho.1)$ holds, a contradiction.
If $(\alpha.1)$ and $(\gamma-)$ hold then $(\rho.0)$ holds.

If $(\alpha.0)$ and ($(\beta+.0)$ or $(\beta-.0)$) hold then $2A\equiv 0$~mod~$n$, so $(\rho.1)$ holds, a contradiction; if $(\alpha.1)$ and ($(\beta+.1)$ or $(\beta-.1)$) hold then $2B\equiv 0$~mod~$n$, so $(\rho.0)$ holds.
If $(\beta+.0)$ and $(\beta-.0)$ hold then $2A\equiv 0$~mod~$n$ so $(\rho.1)$ holds, a contradiction; if $(\beta+.1)$ and $(\beta-.1)$ hold then $2B\equiv 0$~mod~$n$ so $(\rho.0)$ holds.

If $(\beta-.0)$ and $(\gamma-)$ hold then $A\equiv 3B$~mod~$n$ and $4B\equiv 0$~mod~$n$ so $(\rho.0)$ holds, a contradiction. Similarly, if $(\beta-.1)$ and $(\gamma-)$ hold then $(\rho.1)$ holds.

If $(\beta+.0)$ and $(\beta+.1)$ hold then $B\equiv -3A$~mod~$n$ and $8A\equiv 0$~mod~$n$; moreover $1=(n,A,B)=(n,A,-3A)=(n,A)$ so $n|8$, and if $n<8$ then $(\rho.0)$ holds, a contradiction, so $n=8$ and $G_n(x_0x_mx_k^{-1})$ is isomorphic to $H(8,4)$. Similarly, if $(\beta-.0)$ and $(\beta-.1)$ hold then $G_n(x_0x_mx_k^{-1})\cong H(8,6)$.
If $(\beta+.0)$ and $(\gamma-)$ hold then $A\equiv -3B$~mod~$n$ and $8B\equiv 0$~mod~$n$; moreover $1=(n,A,B)=(n,-3B,B)=(n,B)$ so $n|8$, and if $n<8$ then $(\rho.0)$ holds, a contradiction, so $n=8$ and $G_n(x_0x_mx_k^{-1})\cong H(8,4)$. Similarly, if $(\beta+.1)$ and $(\gamma-)$ hold then $n=8$ and $G_n(x_0x_mx_k^{-1})\cong H(8,4)$.

If $(\beta+.0)$ and $(\beta-.1)$ hold then $B\equiv 3A$~mod~$n$ and $10A\equiv 0$~mod~$n$; moreover $1=(n,A,B)=(n,A)$ so $n|10$, and if $n<10$ then $(\rho.0)$ holds, a contradiction, so $n=10$. Then $(k,n)=1$ so by~\cite[Lemma~1.3]{BardakovVesnin} we may assume $k=1$, so $A=1$ and $k-m=B=3$, and hence $m=8$; thus $G_n(x_0x_mx_k^{-1})\cong G_{10}(x_0x_8x_1^{-1})$. Similarly, if $(\beta+.1)$ and $(\beta-.0)$ hold then $n=10$ and $G_n(x_0x_mx_k^{-1})\cong G_{10}(x_0x_4x_1^{-1})$. By~\cite[Theorem~2]{COS} we have $G_{10}(x_0x_8x_1^{-1})\cong G_{10}(x_0x_4x_1^{-1})= H(10,4)$.

If $(\alpha.0)$ and $(\beta+.1)$ hold then $B\equiv -3A$ and $18A\equiv 0$~mod~$n$; moreover $1=(n,A,B)=(n,A)$ so $n|18$ and if $n\leq 9$ then $(\rho.0)$ holds so $n=18$. Then $(k,n)=1$ so we may assume $k=1$ so $k-m=B=-3$ and hence $m=4$ so $G_n(x_0x_mx_k^{-1})\cong G_{18}(x_0x_4x_1^{-1})=H(18,4)$.
If $(\alpha.1)$ and $(\beta+.0)$ hold then $A\equiv -3B$ and $18B\equiv 0$~mod~$n$; moreover $1=(n,A,B)=(n,B)$ so $n|18$ and again $n=18$. Then $(k-m,n)=1$ so we may assume that $B=k-m=1$ so $k=A=-3$, and hence $m=-4$, so $G_n(x_0x_mx_k^{-1})=G_{18}(x_0x_{-4}x_{-3}^{-1})\cong G_{18}(x_0x_4x_1^{-1})=H(18,4)$ by~\cite[Lemma~1.3]{BardakovVesnin}, \cite[Lemma~7]{Wsurvey}.

If $(\alpha.0)$ and $(\beta-.1)$ hold then $B\equiv 3A$ and $18A\equiv 0$~mod~$n$; moreover $1=(n,A,B)=(n,A)$ so $n|18$ and if $n\leq 9$ then $(\rho.0)$ holds so $n=18$. Then $(k,n)=1$ so we may assume $k=1$ so $k-m=B=3$ and hence $m=-2$ so $G_n(x_0x_mx_k^{-1})\cong G_{18}(x_0x_{16}x_1^{-1})=H(18,16)$.
If $(\alpha.1)$ and $(\beta-.0)$ hold then $A\equiv 3B$ and $18B\equiv 0$~mod~$n$; moreover $1=(n,A,B)=(n,B)$ so $n|18$ and if $n\leq 9$ then $(\rho.1)$ holds so $n=18$. Then $(k-m,n)=1$ so we may assume $k-m=1$ so $k=A=3$ and hence $k=3, m=2$ so $G_n(x_0x_mx_k^{-1})\cong G_{18}(x_0x_{2}x_3^{-1})\cong H(18,16)$ by~\cite[Lemma~1.3]{BardakovVesnin}, \cite[Lemma~7]{Wsurvey}.
\end{proof}

In Corollary~\ref{cor:G8nonposnothyp} we showed that $H(8,4)$ is not hyperbolic; in Lemma~\ref{lem:H(n,n/2+2)nothyp} we will show that $H(8,6)$ is not hyperbolic. We now show that the remaining groups arising in Lemma~\ref{lem:nonposmorethanonelength6NewVersion} are hyperbolic:

\begin{example}\label{ex:hypgroupswithn=9,10,18}
\em Using KBMAG~\cite{KBMAG} it is straightforward to show that the groups $H(10,4)$, $H(18,4)$, $H(18,16)$ are hyperbolic, and since they contain a non-abelian free subgroup (by~\cite[Corollary~11]{HowieWilliams}), they are non-elementary hyperbolic.\em
\end{example}

\subsection{Non-hyperbolic groups $G_n(x_0x_mx_k^{-1})$}\label{sec:nonhyperbolicgroupsnonpositive}

We now show that the $T(6)$ groups $H(n,n/2+2)$ arising in Lemma~\ref{lem:2(A-B)=0groups} are not hyperbolic. (Note that if $n=2,4,6$ or $10$ then the presentation of $H(n,n/2+2)$ does not satisfy $T(6)$, by Corollary~\ref{cor:T6T7nonpositive}.) As in the proof of Lemma~\ref{lem:n=8nonhyp} we do this by an application of the Flat Plane Theorem.

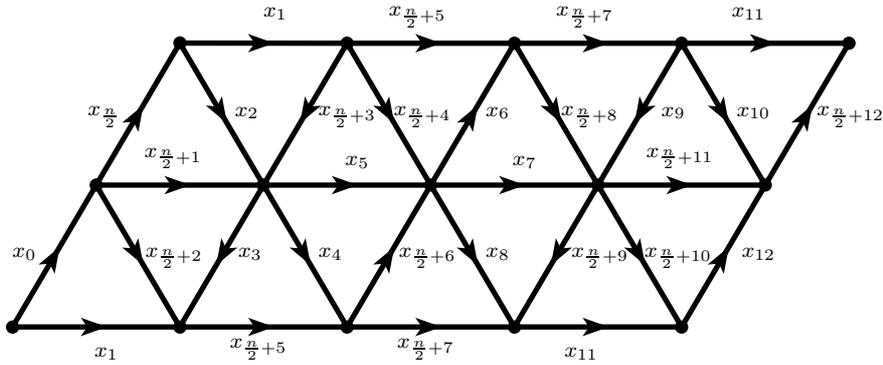
\begin{figure}[ht]
\begin{center}
\psset{xunit=1.1cm,yunit=1.1cm,algebraic=true,dimen=middle,dotstyle=o,dotsize=5pt 0,linewidth=1.6pt,arrowsize=3pt 2,arrowinset=0.25}
\begin{pspicture*}(0.5,-2.2)(12.0,2.5)
\psline[linewidth=2.pt,ArrowInside=->,ArrowInsidePos=0.5](1.,-1.71)(2.,0.)
\psline[linewidth=2.pt,ArrowInside=->,ArrowInsidePos=0.5](2.,0.)(3.,1.712)
\psline[linewidth=2.pt,ArrowInside=->,ArrowInsidePos=0.5](3.,1.712)(5.,1.71)
\psline[linewidth=2.pt,ArrowInside=->,ArrowInsidePos=0.5](5.,1.71)(7.,1.71)
\psline[linewidth=2.pt,ArrowInside=->,ArrowInsidePos=0.5](7.,1.71)(9.,1.71)
\psline[linewidth=2.pt,ArrowInside=->,ArrowInsidePos=0.5](9.,1.71)(11.,1.71)
\psline[linewidth=2.pt,ArrowInside=->,ArrowInsidePos=0.5](1.,-1.71)(3.,-1.71)
\psline[linewidth=2.pt,ArrowInside=->,ArrowInsidePos=0.5](3.,-1.71)(5.,-1.71)
\psline[linewidth=2.pt,ArrowInside=->,ArrowInsidePos=0.5](5.,-1.71)(7.,-1.71)
\psline[linewidth=2.pt,ArrowInside=->,ArrowInsidePos=0.5](7.,-1.71)(9.,-1.71)
\psline[linewidth=2.pt,ArrowInside=->,ArrowInsidePos=0.5](9.,-1.71)(10.,0.)
\psline[linewidth=2.pt,ArrowInside=->,ArrowInsidePos=0.5](10.,0.)(11.,1.71)
\psline[linewidth=2.pt,ArrowInside=->,ArrowInsidePos=0.5](2.,0.)(4.,0.)
\psline[linewidth=2.pt,ArrowInside=->,ArrowInsidePos=0.5](4.,0.)(6.,0.)
\psline[linewidth=2.pt,ArrowInside=->,ArrowInsidePos=0.5](6.,0.)(8.,0.)
\psline[linewidth=2.pt,ArrowInside=->,ArrowInsidePos=0.5](8.,0.)(10.,0.)
\psline[linewidth=2.pt,ArrowInside=->,ArrowInsidePos=0.5](3.,1.712)(4.,0.)
\psline[linewidth=2.pt,ArrowInside=->,ArrowInsidePos=0.5](5.,1.71)(6.,0.)
\psline[linewidth=2.pt,ArrowInside=->,ArrowInsidePos=0.5](7.,1.71)(8.,0.)
\psline[linewidth=2.pt,ArrowInside=->,ArrowInsidePos=0.5](9.,1.71)(10.,0.)
\psline[linewidth=2.pt,ArrowInside=->,ArrowInsidePos=0.5](2.,0.)(3.,-1.71)
\psline[linewidth=2.pt,ArrowInside=->,ArrowInsidePos=0.5](4.,0.)(5.,-1.71)
\psline[linewidth=2.pt,ArrowInside=->,ArrowInsidePos=0.5](6.,0.)(7.,-1.71)
\psline[linewidth=2.pt,ArrowInside=->,ArrowInsidePos=0.5](8.,0.)(9.,-1.71)
\psline[linewidth=2.pt,ArrowInside=->,ArrowInsidePos=0.5](5.,1.71)(4.,0.)
\psline[linewidth=2.pt,ArrowInside=->,ArrowInsidePos=0.5](4.,0.)(3.,-1.71)
\psline[linewidth=2.pt,ArrowInside=->,ArrowInsidePos=0.5](5.,-1.71)(6.,0.)
\psline[linewidth=2.pt,ArrowInside=->,ArrowInsidePos=0.5](6.,0.)(7.,1.71)
\psline[linewidth=2.pt,ArrowInside=->,ArrowInsidePos=0.5](9.,1.71)(8.,0.)
\psline[linewidth=2.pt,ArrowInside=->,ArrowInsidePos=0.5](8.,0.)(7.,-1.71)
\begin{scriptsize}
\psdots[dotstyle=*](1.,-1.71)
\psdots[dotstyle=*](3.,1.712)
\psdots[dotstyle=*](5.,1.71)
\psdots[dotstyle=*](7.,1.71)
\psdots[dotstyle=*](9.,1.71)
\psdots[dotstyle=*](3.,-1.71)
\psdots[dotstyle=*](5.,-1.71)
\psdots[dotstyle=*](7.,-1.71)
\psdots[dotstyle=*](9.,-1.71)
\psdots[dotstyle=*](11.,1.71)
\psdots[dotstyle=*](8.,0.)
\psdots[dotstyle=*](2.,0.)
\psdots[dotstyle=*](4.,0.)
\psdots[dotstyle=*](6.,0.)
\psdots[dotstyle=*](10.,0.)

\rput[bl](1.9,0.7){$x_{\frac{n}{2}}$}
\rput[bl](3.66,0.8){$x_2$}
\rput[bl](4.66,0.7){$x_{\frac{n}{2}+3}$}
\rput[bl](5.56,0.7){$x_{\frac{n}{2}+4}$}
\rput[bl](6.66,0.8){$x_6$}
\rput[bl](7.56,0.7){{$x_{\frac{n}{2}+8}$}}
\rput[bl](8.76,0.8){$x_9$}
\rput[bl](9.66,0.8){$x_{10}$}
\rput[bl](10.62,0.7){$x_{\frac{n}{2}+12}$}
\rput[bl](4.,2.000){$x_{1}$}
\rput[bl](5.5,1.900){$x_{\frac{n}{2}+5}$}
\rput[bl](7.5,1.900){$x_{\frac{n}{2}+7}$}
\rput[bl](9.6,2.000){$x_{11}$}
\rput[bl](2.58,0.2){$x_{\frac{n}{2}+1}$}
\rput[bl](4.98,0.2){$x_5$}
\rput[bl](6.98,0.2){$x_7$}
\rput[bl](8.58,0.2){$x_{\frac{n}{2}+11}$}
\rput[bl](1.0,-0.9){$x_0$}
\rput[bl](2.59,-1.0){$x_{\frac{n}{2}+2}$}
\rput[bl](3.70,-0.9){$x_3$}
\rput[bl](4.66,-0.9){$x_4$}
\rput[bl](5.62,-1.0){{$x_{\frac{n}{2}+6}$}}
\rput[bl](6.66,-0.9){$x_8$}
\rput[bl](7.59,-1.0){{ $x_{\frac{n}{2}+9}$}}
\rput[bl](8.56,-1.0){$x_{\frac{n}{2}+10}$}
\rput[bl](9.72,-0.9){$x_{12}$}
\rput[bl](1.98,-2.1){$x_1$}
\rput[bl](3.6,-2.1){$x_{\frac{n}{2}+5}$}
\rput[bl](5.6,-2.1){$x_{\frac{n}{2}+7}$}
\rput[bl](7.6,-2.1){$x_{11}$}
\end{scriptsize}
\end{pspicture*}
\end{center}
  \caption{A van Kampen diagram over the presentation $P_n(x_0x_{n/2+2}x_1^{-1})$ with boundary label $(x_0x_{n/2}) (x_1x_{n/2+5}x_{n/2+7}x_{11})(x_{12}x_{n/2+12})^{-1} (x_1x_{n/2+5}x_{n/2+7}x_{11})^{-1}$.\label{fig:nonpositiveflat}}
\end{figure}

\begin{lemma}\label{lem:H(n,n/2+2)nothyp}
Suppose $n\geq 8$ is even, $n\neq 10$. Then $H(n,n/2+2)=G_n(x_0x_{n/2+2}x_{1}^{-1})$ is not hyperbolic.
\end{lemma}

\begin{proof}
Since the presentation $P_n(x_0x_{n/2+2}x_1^{-1})$ satisfies C(3)-T(6) and each relator has length~3, each face in the geometric realisation $\tilde{C}$ of the Cayley complex of $P$ (obtained by assigning length~1 to each edge) is an equilateral triangle, and so $\tilde{C}$ satisfies the CAT(0) inequality. Consider the geometric realisation $\Delta_0$ of the reduced van Kampen diagram given in Figure~\ref{fig:nonpositiveflat} and for each $0\leq i<n$ let $\Delta_i$ be obtained from $\Delta_0$ by applying the shift $\theta^i$ to each edge. Then placing $\Delta_0,\Delta_{12},\Delta_{24},\ldots ,\Delta_{6n-12}$ side by side gives the geometric realisation $\Delta$ of a reduced van Kampen diagram. Copies of $\Delta$ tile the Euclidean plane without cancellation of faces. Thus there is an isometric embedding of the Euclidean plane in $\tilde{C}$, and so the result follows from the Corollary to Theorem~A in~\cite{Bridson}.
\end{proof}

\subsection{Analysis of van Kampen diagrams over $P_n(x_0x_mx_k^{-1})$}\label{sec:curvaturenonpositive}

In this section we show that if the cyclic presentation $P=P_n(x_0x_mx_k^{-1})$ is $T(6)$ and precisely one of the
congruences ($\alpha.j$),($\beta+.j$),($\beta-.j$),($\gamma-$) hold then $G = G_n(x_0x_mx_k^{-1})$ is hyperbolic. As in Section~\ref{sec:curvaturepositive} we do this by showing that there is a linear function $f:\mathbb{N}\rightarrow \mathbb{N}$ such that for all $N\in \mathbb{N}$ and all freely reduced words $W\in F_n$ with length at most $N$ that represent the identity of $G$ we have $\mathrm{Area}(W)\leq f(N)$. Note that each face in $D$ is a triangle, as shown in Figure~\ref{fig:nonpositiveface}, where the corner labels $X,Y,Z$ correspond to the edge types of the star graph of $P$. In order to obtain a linear isoperimetric function (in Lemma~\ref{lem:curvaturenonpositivecase}) we first rule out certain configurations in $D$.

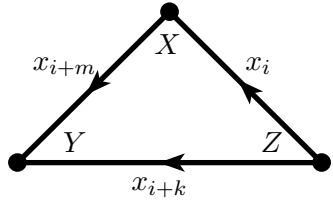
\begin{figure}
\begin{center}
\psset{xunit=1cm,yunit=1cm,algebraic=true,dimen=middle,dotstyle=o,dotsize=7pt 0,linewidth=1.6pt,arrowsize=3pt 2,arrowinset=0.25}
\begin{pspicture*}(9.5,13.5)(16,17)

\psline[linewidth=2pt,ArrowInside=->,ArrowInsidePos=0.5](14,14)(10,14)
\psline[linewidth=2pt,ArrowInside=->,ArrowInsidePos=0.5](14,14)(12,16)
\psline[linewidth=2pt,ArrowInside=->,ArrowInsidePos=0.5](12,16)(10,14)

\rput[tl](11.8,15.7){$X$}
\rput[tl](10.6,14.4){$Y$}
\rput[tl](13.2,14.4){$Z$}
\rput[tl](10.2,15.4){$x_{i+m}$}
\rput[tl](13.0,15.4){$x_i$}
\rput[tl](11.5,13.8){$x_{i+k}$}
\begin{scriptsize}
\psdots[dotstyle=*,linecolor=black](12,16)
\psdots[dotstyle=*,linecolor=black](10,14)
\psdots[dotstyle=*,linecolor=black](14,14)
\end{scriptsize}
\end{pspicture*}
\end{center}
  \caption{A typical face in a van Kampen diagram over the presentation $P_n(x_0x_mx_k^{-1})$.\label{fig:nonpositiveface}}
\end{figure}

\begin{lemma}\label{lem:nonposalldegree6}
Suppose that all interior vertices of $D$ have degree at least 6 and that all interior vertices of degree 6 of $D$ correspond to  precisely one of $(\alpha.j)$, $(\beta+.j)$, $(\beta-.j)$, $(\gamma-)$ for $j\in \{0,1\}$. If $v$ is an interior vertex of degree~6 where all the neighbours of $v$ are interior vertices then $v$ has a neighbour of degree at least 7.
\end{lemma}

\begin{proof}
If $v$ is labelled $Y^6$ (resp. $Z^6$) then clearly none of its neighbours can be labelled $Y^6$ (resp. $Z^6$), so they must each have degree at least 7. If $v$ is labelled $XZXYYY$ (resp. $XYXZZZ$, resp. $XZZXYY$), then the labels of the corners of the faces incident to $v$ show that at least one of the neighbours of $v$ does not have label $XZXYYY$ (resp. $XYXZZZ$, resp. $XZZXYY$), and hence has degree at least 7.
\end{proof}

We are now in a position to be able to establish the existence of a suitable isoperimetric function.

\begin{lemma}\label{lem:curvaturenonpositivecase}
Let $n\geq 2$, $0\leq m,k<n$, $m\neq k, k\neq 0$ and set $A=k,B=k-m$ and let $\Gamma$ be the star graph of $P_n(x_0x_mx_k^{-1})$. Suppose that none of $(\rho.j)$, $(\sigma+.j)$, $(\sigma-.j)$, $(\tau+.j)$, $(\tau-.j)$, hold and that exactly one of the congruences ($\alpha.j$), ($\beta+.j$), ($\beta-.j$), ($\gamma-$) of Table~\ref{tab:nonposshortcycles} holds ($j\in \{0,1\}$). Then $G_n(x_0x_mx_k^{-1})$ has a linear isoperimetric function.
\end{lemma}

\begin{proof}
Let $N\in \mathbb{N}$ and let $W$ be a freely reduced word in the free group $F_n$ of length at most $N$ that represents the identity of $G$ and let $D$ be a reduced van Kampen diagram whose boundary is a simple closed curve with label $W$. We let $I$ denote the set of interior vertices of $D$, $B$ denote the set of boundary vertices of $D$, and $F$ denote the set of faces of $D$. Then $\mathrm{Area}(W)\leq |F|$. Writing $\pi$ to denote $180$, we define the curvature of a face $f$ by $\kappa(f)=-\pi + (\mathrm{sum~of~angles~in}~f)$, the curvature of an interior vertex $v$ by $\kappa (v)=2\pi - (\mathrm{sum~of~angles~at}~v)$, and the curvature of a boundary vertex $\hat{v}$ by $\kappa(\hat{v})=\pi - (\mathrm{sum~of~angles~at}~\hat{v})$. Again it follows from the Gauss-Bonnet Theorem that equation~(\ref{eq:GaussBonnet}) holds.

Since none of the congruences $(\rho.j)$, $(\sigma+.j)$, $(\sigma-.j)$, $(\tau+.j)$, $(\tau-.j)$ hold, every interior vertex of $D$ is of degree at least~6 and since exactly one of the congruences ($\alpha.j$),($\beta+.j$),($\beta-.j$),($\gamma-$) hold, then label of an interior vertex of degree 6 is the corresponding label given in Table~\ref{tab:nonposshortcycles}.

We assign angles to the corners of faces in $D$ as follows. If $v$ is a boundary vertex assign 47 to every corner at $v$; if $v$ is an interior vertex of degree at least $7$ then assign 52 to every corner at $v$. Assume now that $v$ is an interior vertex of degree $6$ and consider a face $f$ with vertices $v$ and $u,w$: if $u,w$ are interior of degree 6 then assign 59 to the corner of $f$ at $v$; otherwise assign $63.5$ to the corner of $f$ at $v$.
 If a face $f$ contains a boundary vertex then $\kappa(f)\leq -\pi +47 +2(63.5)=-6$; if a face $f$ contains only interior vertices of degree 6 then $\kappa(f)=-\pi+3(59)=-3$; if a face $f$ contains only interior vertices of degree at least 7 then $\kappa(f)= -\pi +3(52)=-24$; if a face $f$ contains an interior vertex of degree 6 and two interior vertices of degree at least 7 then $\kappa(f)=-\pi+63.5+2(52)=-12.5$; if a face $f$ contains two vertices of degree 6 and one of degree at least 7 then $\kappa(f)=-\pi+2(63.5)+52=-1$. Therefore $\kappa(f)\leq -1$ for all $f\in F$.

We now turn to curvature of the vertices. If $v$ is an interior vertex of degree at least 7 then $\kappa(v)\leq 2\pi -7(52)=-4$; if $v$ is an interior vertex of degree 6 that has a neighbour that is either interior of degree at least 7 or is a boundary vertex then $\kappa(v)\leq 2\pi -4(59)-2(63.5)=-3$.

By Lemma~\ref{lem:nonposalldegree6} every interior vertex of degree 6 has a neighbour on the boundary or a neighbour that is interior of degree at least 7. Then $\kappa(v)\leq -3$ for all interior vertices $v$ so~(\ref{eq:GaussBonnet}) implies
\begin{alignat*}{1}
  2\pi &= \sum_{v\in I} \kappa(v) + \sum_{\hat{v}\in B} \kappa (\hat{v}) +\sum_{f\in F} \kappa(f)\\
       &\leq \sum_{v\in I} (-3) + \sum_{\hat{v}\in B} (\pi -\mathrm{sum~of~angles~at}~\hat{v}) + \sum_{f\in F} (-1)\\
       &= -3|I| + \sum_{\hat{v}\in B} (\pi -\mathrm{sum~of~angles~at}~\hat{v}) -|F|\\
       &\leq |B|\pi - \sum_{\hat{v}\in B} (\mathrm{sum~of~angles~at}~\hat{v}) -|F|
\end{alignat*}
so
\[ \sum_{\hat{v}\in B} (\mathrm{sum~of~angles~at}~\hat{v}) \leq (|B|-2)\pi -|F|.\]
On the other hand, the corner angle at any boundary vertex is $47$ so the sum of angles over the boundary vertices is bounded below by $47|B|$. Therefore $47|B|\leq (|B|-2)\pi -|F|$ so $|F|\leq 133|B|-360$. But $\mathrm{Area}(W)\leq |F|$ and $|B|\leq N$ so $\mathrm{Area}(W) \leq 133N-360$, so $f(N)=133N-360$ is a linear isoperimetric function, as required.
\end{proof}

We now have all the ingredients to prove Theorem~\ref{thm:nonpositive}.

\subsection{Proof of Theorem~\ref{thm:nonpositive}}\label{sec:proofofnonpositive}

Suppose that $n\geq 2$, $0\leq m,k<n$, $m\neq k$, $k\neq 0$, $(n,m,k)=1$ and that the cyclic presentation $P_n(x_0x_mx_k^{-1})$ satisfies $T(6)$. Then Lemma~\ref{lem:nonposshortcycles} implies that none of the congruences $(\rho.j)$, $(\sigma+.j)$, $(\sigma-.j)$, $(\tau+.j)$, $(\tau-.j)$ hold. If $(\gamma+)$ holds then $n=8$ or $n\geq 12$ and $G$ is not hyperbolic by Lemmas~\ref{lem:2(A-B)=0groups},\ref{lem:H(n,n/2+2)nothyp}, so suppose that $(\gamma+)$ does not hold.

If $\Gamma$ has no cycle of length less than 7, then $P_n(x_0x_mx_k^{-1})$ satisfies C(3)-T(7), so $G_n(x_0x_mx_k^{-1})$ is hyperbolic by~\cite[Corollary~4.1]{GerstenShortI}, so we may assume that $\Gamma$ has a cycle of length~6 so, by Lemma~\ref{lem:nonposshortcycles}, at least one of the congruences
$(\alpha.j)$, $(\beta+.j)$, $(\beta-.j)$, $(\gamma-)$ hold ($j\in \{0,1\}$). Suppose that more than one of them hold. Then $G$ is one of the groups in the conclusion of Lemma~\ref{lem:nonposmorethanonelength6NewVersion}. When $n=8$ the group $G\cong G_8(x_0x_4x_1^{-1})=H(8,4)$ or $G\cong G_8(x_0x_6x_1^{-1})=H(8,6)$, which are non-hyperbolic by Corollary~\ref{cor:G8nonposnothyp} and Lemma~\ref{lem:H(n,n/2+2)nothyp}, respectively. In the remaining cases $G$ is non-elementary hyperbolic by Example~\ref{ex:hypgroupswithn=9,10,18}.

Suppose then that exactly one of the congruences $(\alpha.j)$, $(\beta+.j)$, $(\beta-.j)$, $(\gamma-)$ holds. Then $G$ has a linear isoperimetric function, and hence is hyperbolic, by Lemma~\ref{lem:curvaturenonpositivecase}. By~\cite[Corollary~11]{HowieWilliams} $G$ contains a non-abelian free subgroup so it is non-elementary hyperbolic.

\subsection*{Acknowledgements}

The first named author thanks Jim Howie for his hospitality during a research visit to the School of Mathematical and Computer Sciences, Heriot Watt University, in November 2019; that visit was central to the development of the Flat Plane arguments used in this paper. Both authors thank Jim Howie and William Bogley for insightful conversations relating to this work, and helpful comments on a draft of this article.

  \textsc{Ihechukwu Chinyere: Department of Mathematical Sciences, University of Essex, Wivenhoe Park, Colchester, Essex CO4 3SQ, UK.}\par\nopagebreak
  \textit{E-mail address}, \texttt{ihechukwu.chinyere@essex.ac.uk}

  \medskip

  \textsc{Gerald Williams: Department of Mathematical Sciences, University of Essex, Wivenhoe Park, Colchester, Essex CO4 3SQ, UK.}\par\nopagebreak
  \textit{E-mail address}, \texttt{gerald.williams@essex.ac.uk}

\end{document}